\documentclass[a4paper,11pt]{amsart}
\usepackage[dvipsnames]{xcolor}
\usepackage{amsthm,amsmath,amsfonts,amssymb,tikz,tikz-cd,geometry,enumitem,graphicx,pst-node,bbm,pgfplots,subfigure,hyperref,mathtools}

\mathtoolsset{showonlyrefs}

\graphicspath{{./Figures/}}
\graphicspath{ {./Figures/} }

\newtheorem{thm}{Theorem}[section]
\newtheorem*{thm*}{Theorem}

\newtheorem*{cor*}{Corollary}
\newtheorem{lem}[thm]{Lemma}

\newtheorem{prop}[thm]{Proposition}

\newcommand*{\defeq}{\mathrel{\vcenter{\baselineskip0.5ex \lineskiplimit0pt\hbox{\scriptsize.}\hbox{\scriptsize.}}}=}

\newcommand{\q}{\quad\quad}

\let\originalleft\left
\let\originalright\right
\renewcommand{\left}{\mathopen{}\mathclose\bgroup\originalleft}
\renewcommand{\right}{\aftergroup\egroup\originalright}

\renewcommand{\b}{\underbar}
\renewcommand{\d}{\underdot}

\newcommand{\en}{\textendash\ }

\let\temp\phi
\let\phi\varphi
\let\varphi\temp

\newcommand{\cB}{\mathcal{B}}

\newcommand{\cD}{\mathcal{D}}

\newcommand{\cF}{\mathcal{F}}

\newcommand{\cH}{\mathcal{H}}
\newcommand{\cI}{\mathcal{I}}
\newcommand{\cJ}{\mathcal{J}}

\newcommand{\cN}{\mathcal{N}}

\newcommand{\cP}{\mathcal{P}}

\newcommand{\cS}{\mathcal{S}}
\newcommand{\cT}{\mathcal{T}}
\newcommand{\cU}{\mathcal{U}}
\newcommand{\cV}{\mathcal{V}}

\newcommand{\N}{\mathbb{N}}

\renewcommand{\a}{\alpha}
\renewcommand{\b}{\beta}

\renewcommand{\d}{\delta}

\newcommand{\e}{\varepsilon}

\renewcommand{\k}{\kappa}

\newcommand{\s}{\sigma}

\newcommand{\vs}{\varsigma}
\renewcommand{\th}{\theta}
\newcommand{\Th}{\Theta}
\newcommand{\w}{\omega}
\newcommand{\W}{\Omega}

\newcommand{\fn}{\mathfrak n}

\setenumerate{label={\normalfont(\roman*)}}

\pgfplotsset{compat=1.17}

\usepackage[dvipsnames]{xcolor}
    \definecolor{awesome}{rgb}{1.0, 0.13, 0.32}

\numberwithin{equation}{section}

\begin{document}
\title[Dimension of Besicovitch-Eggleston sets for non-autonomous systems]{Dimension of Besicovitch-Eggleston sets for non-autonomous systems with countable symbolic dynamics}

\date{}
\author[Jonny Imbierski] {Jonny Imbierski$^\dagger$}
\author[Charlene Kalle]{Charlene Kalle$^\ddagger$}

\address[$\dagger$]{Mathematisch Instituut, Leiden University, Einsteinweg 55, 2333CC Leiden, The Netherlands}
\email[Jonny Imbierski]{imbierskijf@math.leidenuniv.nl}
\address[$\ddagger$]{Mathematisch Instituut, Leiden University, Einsteinweg 55, 2333CC Leiden, The Netherlands}
\email[Charlene Kalle]{kallecccj@math.leidenuniv.nl}


%\tableofcontents
\begin{abstract}
    \noindent
    In this article\footnote{Version: \today} 
    we derive a formula for the Hausdorff dimension of Besicovitch-Eggleston level sets associated with non-autonomous dynamics constructed from families of countable affine iterated function systems. The formula obtained shows that the universal-lower-bound phenomenon present in the autonomous case studied in \cite{FLMW10} persists in this non-autonomous setting.
\end{abstract}

\maketitle

\section{Introduction}

A countable iterated function system (IFS) on a compact, connected subset $X\subseteq \mathbb R^D$, $D \ge 1$, is a collection $\{ f_k: X \to X \}_{k \in I}$ of uniformly contracting maps with $I$ a countable index set. %While finite IFSs, with finite index set $I$, have been extensively studied, the study of infinite IFSs, with countably infinite index set $I$, started with the work of Mauldin and Williams \cite{MW86}, see also \cite{Mau95,MU96} (Check falconer 86 random). 
Let $\pi: I^\mathbb N \to X$ be the projection given by
\[ \pi((\omega_n)_{n \in \mathbb N}) = \lim_{n \to \infty} f_{\omega_1} \circ \cdots \circ f_{\omega_n}(X).\]
The {\em limit set} of the IFS is the set $\Lambda = \pi(I^\mathbb N)$. If we use $\sigma: I^\mathbb N \to \mathbb N$ to denote the left shift map, then for a H\"older continuous potential $\Phi: I^{\mathbb N} \to \mathbb R^m$, $m \ge 1$, and $\alpha \in \mathbb R^m$ the set
\[ F_{\Phi}(\alpha)= \pi \left( \left\{\omega \in I^\mathbb N \, : \, \lim_{n \to \infty} \frac1n \sum_{k=0}^{n-1} \Phi(\sigma^k(\omega))=\alpha \right\} \right) \]
is the {\em $\alpha$-level set} for the Birkhoff average of $\Phi$. 

\medskip
Much is known about the multifractal properties of Birkhoff average level sets for finite IFSs, i.e.~with a finite index set $I$. %For infinite IFSs, i.e.~with infinite index set $I$, a multifractal analysis has been given in several settings. In \cite{JK11} the Hausdorff dimension of level sets is considered in case the maps $f_k$ are conformal and the normalisation factor of $\frac1n$ in the Birkhoff averages is replaced by the logarithm of the operator norm of the derivative. In \cite{Ree11} Reeve considered non-conformal infinite IFSs in $\mathbb R^2$ that are of Lalley-Gatzouras-type, i.e., they are of a skew-product form with each $f_k$ an affine contraction and with an additional assumption on the contraction ratios. For such systems he established a conditional variational principle for the level sets of Birkhoff averages. In \cite{KR14} the authors developed a thermodynamic formalism for quasi-multiplicative potentials on $I^\mathbb N$, thus establishing a formula for the Hausdorff dimension of level sets for Birkhoff averages of infinite IFSs consisting of affine maps with typical translation vectors.
%\medskip
The case of infinite $I$ is a bit less developed in general, but there are many results for IFSs on $\mathbb R$ that relate to number expansions with infinite digit sets, see e.g.~\cite{BSS02a,KS07,FLM10,FLMW10,Hof10,KMS12,FJLR15,IJ15,GL16,Rush}. One potential that is of particular interest in this setting is $\mathbbm 1: I^\mathbb N \to \mathbb R^{\# I}$ given by $\mathbbm 1 = (\mathbbm 1_{[k]})_{k \in I}$, where $[k] = \{ (\omega_n)_{n \in \mathbb N} \in I^\mathbb N \, : \, \omega_1=k\}$, as this potential captures digit frequencies.
The level sets for the Birkhoff averages corresponding to this potential are also called {\em Besicovitch-Eggleston sets}, referring to the pioneering work of Besicovitch \cite{Bes35} and Eggleston \cite{Egg49} on level sets for digit frequencies for expansions in integer base $N \in \mathbb N_{\ge 2}$.

\medskip
{\em L\"uroth expansions}, which were first introduced by L\"uroth in \cite{Lur83}, are examples of number expansions with infinite digit sets. For irrational numbers $x \in [0,1]$ they are expressions of the form
\[ x = \sum_{n \in \mathbb N} \frac{b_n}{\prod_{k=1}^n b_k(b_k+1)},\]
where $b_n \in \mathbb N$ for each $n \in  \mathbb N$. The sequence $(b_n)_{n \in \mathbb N}$ is called the {\em digit sequence} of $x$. The {\em digit frequency} of the digit $b \in \mathbb N$ in the L\"uroth expansion of $x$ is then
\[ \tau_{L,b}(x)=\lim_{n \to \infty} \frac{\{1 \le k \le n \, : \, b_k=b\}}{n}.\]
Now consider the IFS $\mathcal T_L=\{ f_k:[0,1]\to [0,1] \}_{k \in \mathbb N}$ with $f_k(x)= \frac{x+k}{k(k+1)}$, see Figure~\ref{GLSs}(c). Any sequence $(\omega_n)_{n \in \mathbb N} \in \mathbb N^\mathbb N$ corresponds to a L\"uroth expansion by
\[ \lim_{n \to \infty} f_{\omega_1} \circ \cdots \circ f_{\omega_n}([0,1]) = \sum_{n \ge 1} \frac{\omega_n}{\prod_{k=1}^n \omega_k(\omega_k+1)}.\]
If we let $\alpha = (\alpha_k)_{k \in \mathbb N}$ be a probability vector, i.e.~a vector such that $\alpha_k \ge 0$ and $\sum_{k \in \mathbb N} \alpha_k =1$ (which we call a {\em frequency vector} whenever it refers to digit frequencies), then the set $F_{\mathbbm 1} (\alpha)$ contains those points $x$ for which $\tau_{L,k}(x) = \alpha_k$ for each $k \in \mathbb N$. The result from \cite[Theorem 1.1]{FLMW10} implies the following.

\begin{thm}[\cite{FLMW10}]\label{FLMWThm}
Let $\a=(\a_k)_{k\in\N}$ be a frequency vector. Then the Hausdorff dimension of $F_{\mathcal T_L,\mathbbm 1}(\a)$ is
    \[
        \dim_H F_{\mathcal T_L, \mathbbm 1}(\a)=\max\left\{\frac12,\,\liminf_{m\to+\infty}\frac{-\sum_{1\leq k\leq m}\a_k\log\a_k}{\sum_{1\leq k\leq m}\a_k\log (k(k+1))}\right\}.
    \]
\end{thm}

One of the important observations made in \cite{FLMW10} is the lower bound of $\frac12$ in the formula from Theorem~\ref{FLMWThm}, which is independent of $\alpha$ and occurs due to the fact that the IFS $\mathcal T_L$ is infinite. This $\frac12$ is the \textit{exponent of convergence} of the sequence $(\frac1{k(k+1)})_{k\in\N}$ of contraction ratios of the maps $f_k$ in $\mathcal T_L$: $$\frac12=\inf\left\{t\geq0:\sum_{k\in\N}\left(\frac1{k(k+1)}\right)^t<+\infty\right\}.$$
In the second term in the formula from Theorem~\ref{FLMWThm}, we recognise the entropy and the Lyapunov exponent for the $\alpha$-Bernoulli measure on $\N^\mathbb N$.

\medskip
The main result from \cite{FLMW10} is actually stronger than stated in Theorem~\ref{FLMWThm} in that it holds for a much wider class of infinite IFSs on $\mathbb R$. In case all maps in such an infinite IFS are affine, then the IFS relates to Generalised L\"uroth Series (GLS) expansions, introduced in \cite{BBDK94}, in much the same way as the IFS $\mathcal T_L$ relates to L\"uroth expansions, and the results from \cite{FLMW10} give the Hausdorff dimension of the Besicovitch-Eggleston level sets for such GLS number expansions. See Figure~\ref{GLSs}(a) for a finite IFS that generates binary expansions, Figure~\ref{GLSs}(b) for a finite IFS generating GLS expansions with a finite digit set and Figure~\ref{GLSs}(d) for an infinite IFS generating GLS expansions with an infinite digit set.
\medskip

\begin{figure}[t]
    \centering
    \begin{subfigure}[Binary]{
        \begin{tikzpicture}[scale=3.2]
        \draw[dashed](0,0.5)--(1,0.5);
        \draw[ultra thick, NavyBlue!80](0,0)--(1,0.5)(0,0.5)--(1,1);
        \draw[ultra thick](0,0)node[below]{\small 0}--(1,0)node[below]{\small 1}--(1,1)--(0,1)node[left]{\small 1}--cycle;
        \draw                           (0,1/2)node[left]{\scriptsize $\frac12$};
        \end{tikzpicture}}
    \end{subfigure}
    \hspace{-0.5cm}
    \begin{subfigure}[Finite digit set GLS]{
        \begin{tikzpicture}[scale=3.2]
            \draw[dashed]                   (0,0.3183)--(1,0.3183)(0,5/6)--(1,5/6);
            \draw[ultra thick, NavyBlue!80]   (0,0)--(1,0.3183)(1,0.3183)--(0,5/6)(0,5/6)--(1,1);
            \draw[ultra thick]              (0,0)node[below]{\small 0}--(1,0)node[below]{\small 1}--(1,1)--(0,1)node[left]{\small 1}--
            cycle;
            \draw(0,0.3183)node[left]{\scriptsize $\frac1\pi$}
            (0,5/6)node[left]{\scriptsize $\frac56$};
        \end{tikzpicture}}
    \end{subfigure}
    \hspace{-0.5cm}
    \begin{subfigure}[L\"uroth]{
        \begin{tikzpicture}[scale=3.2]
            \draw[dashed](0,1/5)--(1,1/5)(0,1/4)--(1,1/4)(0,1/3)--(1,1/3)(0,1/2)--(1,1/2);
            \draw[ultra thick, NavyBlue!80](0,1/10)--(1,1/9)(0,1/9)--(1,1/8)(0,1/8)--(1,1/7)(0,1/7)--(1,1/6)(0,1/6)--(1,1/5)(0,1/5)--(1,1/4)(0,1/4)--(1,1/3)(0,1/3)--(1,1/2)(0,1/2)--(1,1);
            \filldraw[fill=NavyBlue!80, draw=NavyBlue!80] (0,0) rectangle (1,1/10);
            \draw[ultra thick](0,0)node[below]{\small 0}--(1,0)node[below]{\small 1}--(1,1)--(0,1)node[left]{\small 1}--cycle;
            \draw(0,1/3)node[left]{\scriptsize $\frac13$}(0,1/2)node[left]{\scriptsize $\frac12$}(-0.01,1/7)node[left]{\scriptsize $\vdots$};
        \end{tikzpicture}}
    \end{subfigure}
    \hspace{-0.5cm}
    \begin{subfigure}[Infinite digit set GLS]{
        \begin{tikzpicture}[scale=3.2]
            \draw[dashed](0,1/2-1/3)--(1,1/2-1/3)(0,1/2-1/4)--(1,1/2-1/4)(0,1/2-1/5)--(1,1/2-1/5)(0,1/2+1/5)--(1,1/2+1/5)(0,1/2+1/4)--(1,1/2+1/4)(0,1/2+1/3)--(1,1/2+1/3);
            \draw[ultra thick, NavyBlue!80] (1,0)--(0,1/2-1/3)(1,1/2-1/3)--(0,1/2-1/4)(1,1/2-1/4)--(0,1/2-1/5)(1,1/2-1/5)--(0,1/2-1/6)(1,1/2-1/6)--(0,1/2-1/7)(1,1/2-1/7)--(0,1/2-1/8)(1,1/2-1/8)--(0,1/2-1/9)(1,1/2-1/9)--(0,1/2-1/10)(0,1/2+1/10)--(1,1/2+1/9)(0,1/2+1/9)--(1,1/2+1/8)(0,1/2+1/8)--(1,1/2+1/7)(0,1/2+1/7)--(1,1/2+1/6)(0,1/2+1/6)--(1,1/2+1/5)(0,1/2+1/5)--(1,1/2+1/4)(0,1/2+1/4)--(1,1/2+1/3)(0,1/2+1/3)--(1,1);
            \fill[NavyBlue!80] (0,1/2-1/9) rectangle (1,1/2+1/9);
            \draw[ultra thick]  (0,0)node[below]{\small $0$}--(1,0)node[below]{\small $1$}--(1,1)--(0,1)node[left]{\small $1$}--cycle;
            \draw   (0,1/6)node[left]{\scriptsize $1/6$}
                    (0,1/4)node[left]{\scriptsize $1/4$}
                    %(3/10,0.01)node[below]{\scriptsize $\frac3{10}$}
                    (-0.01,0.55)node[left]{\scriptsize $\vdots$}
                    %(7/10,0.01)node[below]{\scriptsize $\frac7{10}$}
                    (0,3/4)node[left]{\scriptsize $3/4$}
                    (0,5/6)node[left]{\scriptsize $5/6$};
        \end{tikzpicture}}
    \end{subfigure}
    \caption{Examples of finite and infinite IFSs generating GLS expansions. Binary expansions (a) and L\"uroth expansions (c) are particular examples of GLS expansions.}
    \label{GLSs}
\end{figure}
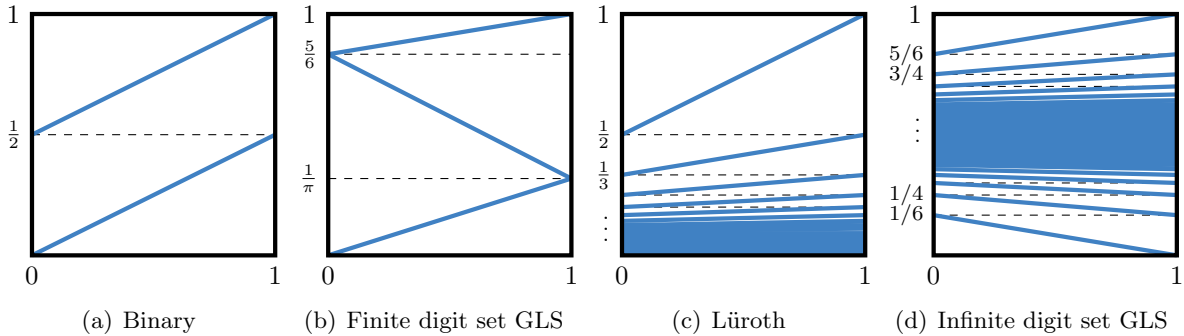

In this paper, we obtain a result similar to Theorem~\ref{FLMWThm} but for the Besicovitch-Eggleston level sets $F_{\cT,\w}(\a)$ of non-autonomous {\bf--} that is, time-dependent {\bf--} IFSs. Let $\mathcal S$ be a finite set and for each $s \in \mathcal S$ let $\mathcal T_s = \{ f_{s,b}:[0,1]\to [0,1]\}_{b \in \mathcal B_s}$ be an IFS with countable index set $\mathcal B_s$ that satisfies the Open Set Condition and for which each $f_{s,b}$ is an affine map such that $\sum_{b \in \mathcal B_s} |f_{s,b}([0,1])|=1$. So, each IFS from Figure~\ref{GLSs} satisfies these conditions. We call an IFS satisfying these conditions a {\em GLS IFS}, since it gives GLS expansions of numbers in $[0,1]$. For notational convenience, we let $\mathcal B_s =\{1,2, \ldots, B_s\}$ for some $B_s \in \mathbb N$ if $\# \mathcal B_s < \infty$ and $\mathcal B_s=\mathbb N$ otherwise.

\medskip
Let $\mathcal T = \{ \mathcal T_s\}_{s \in \mathcal S}$ and fix an $\omega = (\omega_n)_{n \in \mathbb N} \in \mathcal S^\mathbb N$. This $\omega$ determines the order in which the IFSs $\mathcal T_s$ from $\mathcal T$ are applied in the following way. Set $\mathcal B_{\omega}^\mathbb N = \{ (b_n)_{n \in \mathbb N} \, : \, b_k \in \mathcal B_{\omega_k}, \, k \in \mathbb N\}$ and let
\begin{equation}\label{q:piomega}
\pi_\omega: B_{\omega}^\mathbb N \to [0,1], \, (b_n)_{n \in \mathbb N} \mapsto  \lim_{n \to \infty} f_{\omega_1,b_1} \circ f_{\omega_2,b_2} \circ \cdots \circ f_{\omega_n,b_n}([0,1]).
\end{equation}
The conditions on $\mathcal T$ imply that $[0,1] \setminus \pi_\omega(\mathcal B_\omega^\mathbb N)$ is a countable set. For $s \in \mathcal S$ and $b \in \mathcal B_s$ write $N_{s,b}^{-1} = |f_{s,b}([0,1])|$ for the contraction ratio of $f_{s,b}$ and $a_{s,b}=f_{s,b}(0)N_{s,b}$, so that
\[ f_{s,b}(x) = \frac{a_{s,b}+(-1)^{\varepsilon_{s,b}}x}{N_{s,b}},\]
where $\varepsilon_{s,b}\in\{0,1\}$ determines the sign of the slope of $f_{s,b}$.
Then for each $x \in \pi_\omega (\mathcal B_\omega^\mathbb N)$ there is a sequence $(b_n)_{n \in \mathbb N} \in \mathcal B_\omega^\mathbb N$ such that
\begin{equation}\label{NGLSExp}
x = \pi_\omega((b_n)_{n \in \mathbb N}) = \sum_{n\in \mathbb N} \frac{(-1)^{\varepsilon_{\omega_1,b_1}+ \cdots +\varepsilon_{\omega_{n-1},b_{n-1}}}a_{\omega_n,b_n}}{\prod_{k=1}^n N_{\omega_k,b_k}}.
\end{equation}
We call such an expression a {\em Non-autonomous Generalised L\"uroth Series (NGLS) expansion} or $(\mathcal T, \omega)$-expansion of $x$. We refer to the pair $(\cT,\w)$ as a \textit{Non-autonomous Generalised L\"uroth Series (NGLS) system}.
From \eqref{NGLSExp}, we see that an NGLS expansion of $x$ is completely determined by the sequence $(\w_n,b_n)_{n\in\N}$. Therefore, we refer to $\cD\defeq\{(s,b):b\in\cB_s,s\in\cS\}$ as the \textit{digit set} of $(\cT,\w)$ and its elements as \textit{digits} of NGLS expansions.

\medskip
Our main results are on the Hausdorff dimension of the Besicovitch-Eggleston level sets for NGLS systems $(\cT,\w)$; therefore, for $\omega = (\omega_n)_{n \in \mathbb N} \in \mathcal S^\mathbb N$ and a frequency vector $\a=(\a_d)_{d\in\cD}$, set
\[ E_\omega (\alpha) \defeq \left\{ (b_n)_{n \in \mathbb N} \in \mathcal B^\mathbb N_\omega \, : \, \lim_{n \to+\infty} \frac{\# \{ 1\le k \le n : (\w_k, b_k) = d \}} n=\a_d\, \, \text{ for all } \, d\in\cD \right\} \]
for the symbolic level set and let the \textit{Besicovitch-Eggleston $\a$-level set for $(\cT,\w)$} be given by
\[F_{\cT,\w}(\a)\defeq \pi_\omega (E_\omega(\alpha)).\]
Set $\W_\cT(\a)\defeq\{\w\in\cS^\N:F_{\cT,\w}(\a)\neq\emptyset\}$.
Throughout the article, we assume the following non-degeneracy condition on $\a$:
\begin{equation}\label{NonDegenerate}\tag{$\dagger$}
    \textrm{For each }s\in\cS,\textrm{ there exists a }b\in\cB_s\textrm{ such that }\a_{(s,b)}>0.
\end{equation}
The first result is on the form of the set $\W_\cT(\a)$. % and is a generalisation of \cite[Proposition 3.3]{IKM23}.
For each $s\in\cS$, put $\a_s\defeq\sum_{b\in\cB_s}\a_{(s,b)}$.

\begin{thm}\label{Spectrum}
For any finite collection $\cT\defeq\{\mathcal T_s\}_{s\in\cS}$ of GLS IFSs and any frequency vector $\a=(\a_d)_{d\in\cD}$ satisfying \eqref{NonDegenerate}, it holds that
\[ \W_\cT(\a)=\left\{\w\in\cS^\N: \lim_{n \to+\infty}\frac{\# \{1 \le k \le n \, : \, \omega_k = s \}} n=\a_s\, \, \text{ for all }\,  s\in\cS\right\}.\]
In particular, $\Omega_{\mathcal T}(\alpha) \neq \emptyset$.
\end{thm}

For each $s\in\cS$, define the \textit{exponent of convergence} of the system $\mathcal T_s$ by
\[ \eta(\mathcal T_s)\defeq\inf\left\{t\geq0:\sum_{b\in\cB_s}N_{s,b}^{-t}<+\infty\right\}\in[0,1],\]
and put $\eta_\cT\defeq\max\{\eta(\mathcal T_s):s\in\cS\}$. In addition, define the \textit{fibre dimension} with respect to $\cT$ and $\a$ by
\[ \b_\cT(\a)\defeq\liminf_{m\to+\infty}\frac{\sum_{s\in\cS}\a_s\log\a_s-\sum_{d\in\cD_m}\a_d\log\a_d}{\sum_{d\in\cD_m}\a_d\log N_d},\]
where $\cD_m\defeq\{(s,b)\in\cD:b\leq m\}$.
Let $\cS_\N\defeq\{s\in\cS:\cB_s=\N\}$ denote the collection of indices $s \in \mathcal S$ for which the IFS $\mathcal T_s$ is infinite.
\begin{thm}\label{dimH}
    Let $\cT\defeq\{\mathcal T_s\}_{s\in\cS}$ be a finite collection of GLS IFSs, $\a=(\a_d)_{d\in\cD}$ a frequency vector satisfying \eqref{NonDegenerate} and $\w\in\W_\cT(\a)$.
    Assume that the limit
    \begin{equation}\label{q:limexists}
        \lim_{n\to+\infty}\frac{\log n}{\log N_{s,n}}
    \end{equation}
    exists for all $s\in\cS_\N$.
    Then
    \begin{equation}\label{dimHFEq}
        \dim_HF_{\cT,\w}(\a)=\max\{\eta_\cT,\b_\cT(\a)\}.
    \end{equation}
    Furthermore, if $\sum_{d\in\cD}\a_d\log N_d=+\infty$, then $\dim_HF_{\cT,\w}(\a)=\eta_\cT$.
\end{thm}

Comparing \eqref{dimHFEq} with the statement from Theorem~\ref{FLMWThm}, one can see the effect of the non-autonomous setting in both terms. The first term is now given by the maximal exponent of convergence over all systems in $\mathcal T$. In the second term, the measure-theoretic entropy of the $\alpha$-Bernoulli measure appearing in the numerator of the expression from Theorem~\ref{FLMWThm} is now replaced by the fibre entropy of that measure for the non-autonomous system. %; see Section~\ref{Prelim}.
\medskip

The limit in condition \eqref{q:limexists}, when it exists, is in fact an expression for $\eta(\mathcal T_s)$ \en see \eqref{q:etalim} in Section~\ref{sec:Eta}.
The condition is quite unrestrictive and holds, for example, if $n\mapsto N_{s,n}$ is a regularly varying function (in the sense of \cite{Bing}) for each $s\in\cS_\N$. %{\color{awesome} What does it mean for a function on $\mathbb N$ to be regularly varying? What definition did you have in mind?}\ji{I want $N_{s,n}$ to satisfy $N_{s,n}=n^{p_s}\ell_s(n)$ as $n\to+\infty$ for some exponent $p_s>0$ and slowly varying function $\ell_s$ so that the limit in (1.2) is equal to $p_s^{-1}$.If we extend $n\mapsto N_{s,n}$ to a step-function on $\R$ in the natural way, then we can use the definition of regular variation on $\R$, i.e. that $N_{s,x}=x^{p_s}\ell_s(x)$ as $x\to+\infty$.}{\color{awesome}To where in \cite{Bing} are you referencing precisely?}\ji{The unlabelled definition in section 1.4.2 for regular variation combined with Theorem 1.4.1(iii) (Characterisation Theorem) on the previous page.}
\medskip

We can obtain Theorem~\ref{FLMWThm} as a subcase of Theorem~\ref{dimH} by taking $\cT=\{\mathcal T_L\}$ to contain only the L\"uroth system $\mathcal T_L$ introduced above (so $\cD=\{L\}\times\N$).
Then any frequency vector $\a=(\a_d)_{d\in\cD}$ automatically satisfies $(\dagger)$ since $\a_L\defeq\sum_{n\in\N}\a_{(L,n)}=1>0$.
In addition, $N_{L,n}=n(n+1)$ for all $n \in \mathbb N$, so $\lim_{n \to + \infty} \frac{\log n}{\log N_{L,n}}=\frac12$. Theorem~\ref{dimH} also includes the case where $\cT$ contains only finite GLS IFSs, in which case $\mathcal S^\mathbb N = \emptyset$ and \eqref{q:limexists} automatically holds. Moreover, in that case the digit set $\cD$ is finite, so $\eta(\mathcal T_s)=0$ for each $s\in\cS$ and \eqref{dimHFEq} becomes
$$\dim_HF_{\cT,\w}(\a) 
    =\frac{\sum_{s\in\cS}\a_s\log\a_s-\sum_{d\in\cD}\a_d\log\a_d}{\sum_{d\in\cD}\a_d\log N_d}.$$
In \cite[Section 3.3]{IKM23}, this result is already given for most $\omega$ and under additional assumptions. Therefore, Theorem~\ref{dimH} generalises this result.
\medskip

Other generalisations of \cite[Theorem~1.1]{FLMW10} include \cite{GL16} to group frequencies of digits in L\"uroth expansions and \cite{Rush} to countable symbolic dynamical systems modelled by subshifts of finite type. These are both in the autonomous setting. A non-autonomous result similar in flavour to Theorem~\ref{dimH} is that of \cite{NT}, which treats irrational numbers whose semi-regular continued fraction expansions satisfy certain growth conditions. On the other hand, there are many results in the direction of random dynamics for systems related to number expansions including the study of random continued fraction expansions in \cite{KKV17,DO18,KMTV22,BDKKKT24}, random binary expansions in \cite{DK20}, random $\b$-expansions in \cite{DK03,DV05,DK07, DV07,Kem14,BD17,Suz19,Tie23} and random L\"uroth expansions in \cite{KM22a,KM22b,vGKKS}. From the point of view of random dynamics these are all annealed results, while the results of this article are more of a quenched nature.

\medskip
Finally, we remark that the results in this article can be related to IFSs in $\mathbb R^2$ similar to the generalisations of Lalley-Gatzouras carpets studied in e.g.~\cite{Ree11,KR14}. Writing $\mathcal S = \{s_0, \ldots, s_{M-1}\}$ for some integer $M$, we can consider instead of $\mathcal T$ an affine IFS $\mathcal G = \{ g_d \}_{d \in \mathcal D}$ in $\mathbb R^2$ given by $g_d(v,x)=g_{s_i,b}(v,x) = (h_i(v), f_{s_i,b}(x))$, where $\{ h_i\}_{0 \le  i \le M-1}$ is some finite GLS IFS. (For example, taking $h_i(v)=\frac{v+i}{M}$ would work.) Then the set $[0,1]^2 \setminus \pi(\mathcal D^\mathbb N)$ has zero Lebesgue measure. To all but a countable set of points $v \in [0,1]$ there corresponds a unique sequence $\omega = (s_{i_n})_{n \in \mathbb N} \in \mathcal S^\mathbb N$ and for such $\omega$ the Besicovitch-Eggleston $\alpha$-level set $F_{\mathcal T, \omega}(\alpha)$ for $(\mathcal T, \omega)$ can then be obtained by slicing the level set
\[ F_{\mathcal G, \mathbbm 1}(\alpha) = \pi \left( \left\{ (\omega_n,b_n)_{n \in \mathbb N} \in \mathcal D^\mathbb N \, : \, \lim_{n \to \infty} \frac{\{ 1 \le k \le n \, : \, (\omega_k,b_k)=d\}}{n}=\alpha_d \quad \text{for all } d \in \mathcal D \right\} \right) \]
vertically at the position $v$. Under the conditions of Theorem~\ref{dimH} we can thus determine the Hausdorff dimension of these vertical slices.

\medskip
The article is organised as follows. In Section~\ref{Prelim}, we introduce some notation and collect some results on Hausdorff dimension and Besicovitch-Eggleston sets that we will use later. This section also contains the proof of Theorem~\ref{Spectrum}.
Section~\ref{UpperBoundSec} contains the proof of the upper bound of Theorem~\ref{dimH} in Proposition~\ref{p:dimHUpper}. The proof of the lower bound of $\eta_\cT$ in Theorem~\ref{dimH} is given in Section~\ref{sec:Eta} \en see Proposition~\ref{p:dimHLowerEta}. %, while the proofs of the lower bound $\b_\cT(\a)$ and the final statement in Theorem~\ref{dimH} are given in Section~\ref{sec:beta1} (see Proposition~\ref{p:dimHLowerBeta}).
The remainder of the proof of Theorem~\ref{dimH} is given in Section~\ref{sec:beta1}.
In particular, the proof of the lower bound of $\b_\cT(\a)$ is given in Proposition~\ref{p:dimHLowerBeta}.

\section{Preliminaries}\label{Prelim}

In this section, we fix some notation and collect some preliminaries that are used throughout the article.
For any subset $U \subseteq \mathbb R$, we use the notation $|U|= \sup \{ |x-y| \, : \, x,y \in U\}$ to denote its diameter. With $U^\circ$ we denote the interior of $U$.

\subsection{Digit Frequencies}
Let $\mathcal U$ be a countable set of symbols. We use $\mathcal U^\mathbb N$ to denote the set of one-sided infinite sequences $(u_n)_{n \in \mathbb N}$ with $u_n \in \mathcal U$ for all $n \in \mathbb N$. For any $u \in \mathcal U$, $n \in \mathbb N$ and sequence $\zeta = (\zeta_n)_{n \in \mathbb N}
 \in \mathcal U^\mathbb N$, we set
\[ \tau_u(\zeta, n) \defeq \# \{ 1 \le \ell \le n \, : \, \zeta_\ell=u\}\]
for the number of times that the symbol $u$ occurs in the first $n$ elements of $\zeta$, and we let
\[ \tau_u(\zeta) \defeq \lim_{n \to+\infty} \frac{\tau_u(\zeta,n)} n\]
denote the {\em frequency} of the symbol $u$ in the sequence $\zeta$ if the limit exists.

\subsection{Iterated Function Systems and Dimensions of Sets and Measures}
A map $g:[0,1] \to [0,1]$ for which a constant $0 < r< 1$ exists such that $|g(x)-g(y)| \le r|x-y|$ for all $x,y \in [0,1]$ is a {\em contraction} with contraction ratio $r$. Let $I$ be a countable index set. A family $\mathcal G = \{ g_i: [0,1] \to [0,1]\}_{i \in I}$ of contractions on $[0,1]$ is called an {\em iterated function system (IFS)}. It follows by the Cantor Intersection Theorem that for any sequence $(i_n)_{n \in \mathbb N} \in I^\mathbb N$ the set $\bigcap_{n \in \mathbb N} g_{i_1} \circ \cdots \circ g_{i_n}([0,1])$ is a singleton, which we denote by $\pi((i_n)_{n \in \mathbb N})$. The map $\pi: I^\mathbb N \to [0,1]$ defined in this way is the {\em projection map} of $\mathcal G$ and the set $\pi([0,1])$ is called the {\em limit set} of $\mathcal G$. The IFS $\mathcal G$ satisfies the {\em open set condition (OSC)} if there exists a non-empty open subset $U\subseteq [0,1]$ such that $g_i(U) \subseteq U$ and $g_i(U) \cap g_j(U)=\emptyset$ for all distinct $i, j \in I$. 

\medskip
For any $t \ge 0$ and any subset $A\subseteq [0,1]$, we denote the {\em $t$-Hausdorff outer-measure} of $A$ by
$$\cH^t(A)\defeq\lim_{\d\to0}\inf\left\{\sum_{j\in\cJ}|U_j|^t:A\subset\bigcup_{j\in\cJ}U_j,|U_j|<\d\, \, \,  \forall j\in\cJ\right\}.$$
The {\em Hausdorff dimension} of $A$ is then $\dim_H A\defeq\inf\{ t \ge 0 : \cH^t(A)=0\}$.
Fix a Borel probability measure $\mu$ on $[0,1]$. The {\em Hausdorff dimension of $\mu$} is defined by $$\dim_H\mu\defeq\inf\left\{\dim_H A\, :\, A\subset [0,1] \text{  a Borel set},\,\mu(A)=1\right\}.$$
Equivalently, we can also write (e.g.\ see \cite[p41]{Pes97})
\begin{equation}\label{q:dimHmu}
    \dim_H\mu=\lim_{\e\to0}\inf\left\{\dim_H A \, :\, A\subset [0,1] \text{  a Borel set},\,\mu(A)>1-\e\right\},
\end{equation}

The ({\em lower}) {\em pointwise dimension of $\mu$ at $x\in[0,1]$} is defined by $$\underline{d}_\mu(x)\defeq\liminf_{r\to0}\frac{\log\mu\big(B(x,r)\big)}{\log r},$$
where $B(x,r)$ denotes the closed ball of radius $r$ centred at $x$.
We use the following result from \cite{Young} (see also \cite[Theorem 7.1]{Pes97}) to relate lower bounds of the Hausdorff dimension and pointwise dimension of measures.
\begin{lem}\label{l:Pes}
    Let $\mu$ be a Borel probability measure on $[0,1]$.
    If $\underline{d}_\mu(x)\geq d$ for some constant $d\geq0$ and $\mu$-a.e.\ $x\in [0,1]$, then $\dim_H\mu\geq d$.
\end{lem}

\subsection{Definition of Non-autonomous GLS Systems}\label{PrelimNGLSDef}

Throughout the text, we fix the following system. Let $\cS$ be a finite set and let $\cT=\{\mathcal T_s\}_{s\in\cS}$ be a family of {\em GLS IFSs}, meaning that the following conditions hold:
\begin{itemize}
\item each system $\mathcal T_s = \{ f_{s,b}:[0,1]\to [0,1]\}_{s \in \mathcal B_s}$ has a countable index set $\mathcal B_s \subseteq \mathbb N$;
\item the IFS $\mathcal T_s$ satisfies the OSC;
\item each map $f_{s,b}$ is an affine contraction;
\item $\sum_{b \in \mathcal B_s} |f_{s,b}([0,1])|=1$ for each $s$.
\end{itemize}
Writing $N_{s,b}^{-1}$ for the contraction ratio of $f_{s,b}$, then these conditions imply that $\sum_{b \in \mathcal B_s} N_{s,b}^{-1}=1$ for each $s \in \mathcal S$. We let $\cB_s\defeq\N$ in case $\cB_s$ is an infinite set, and, otherwise, we let $\cB_s\defeq\{1,2,\dotsc,B_s\}$ for some $B_s \in \N$.
%We assume throughout that there is at least one $s \in \cS$ for which $\cB_s=\N$.
We assume that the $b \in \mathcal B_s$ are ordered by the contraction ratios for the maps $f_{s,b}$ so that, for $b, b' \in \cB_s$ with $b < b'$, it holds that $N_{s,b}\le N_{s,b'}$.

\medskip
Set $\Omega \defeq \mathcal S^\mathbb N$. For each $\omega \in \Omega$ and $n \in \mathbb N$ let $\cB_\w^n\defeq\prod_{1\leq\ell\leq n}\cB_{\w_\ell}$ denote the collection of all blocks $b_1 \dotsb b_n$ with the property that $b_\ell \in \cB_{\w_\ell}$ for all $1 \le \ell \le n$. For $b_1 \dotsb b_n \in \mathcal B_\omega^n$ we use the notation
\[ \langle b_1 \dotsb b_n \rangle_\w \defeq f_{\omega_1, b_1} \circ \cdots \circ f_{\omega_n,b_n}([0,1]),\]
which is a closed interval that we call a {\em level-$n$ fibre fundamental interval (FFI)}. We denote the collection of all level-$n$ FFIs by $\mathcal F(n)$. Note that for each $n \in \N$, $\langle b_1 \dotsb b_n \rangle_\w^\circ \cap \langle c_1 \dotsb c_n \rangle^\circ_\w = \emptyset$ whenever $b_1 \dotsb b_n  \neq c_1 \dotsb c_n$ and that
\begin{equation}\label{q:lebfibcyl}
\left|\langle b_1\dotsb b_n\rangle_\w \right|=\prod_{1\leq\ell\leq n}\frac1{N_{\w_\ell, b_\ell}}\quad\text{ and }\quad
    \left| \bigcup_{b_1 \dotsb b_n\in\cB_\w^n}\langle b_1\dotsb b_n\rangle_\w \right|=1.
\end{equation}
Since $\langle b_1 \dotsb b_n \rangle_\w \subseteq \langle b_1 \dotsb b_{n-1} \rangle_\w$ for each $n$ and $\lim\limits_{n \to+\infty} \left| \langle b_1 \dotsb b_n \rangle_\w \right| = 0$,
the collection
\begin{equation}\label{q:FFISet}
    \cF\defeq\bigcup_{n\in\N}\cF(n)=\big\{ \langle b_1 \dotsb b_n \rangle_\w : b_1 \dotsb b_n \in \cB_\w^n, \, n \in \N \big\}
\end{equation}
generates the Borel $\s$-algebra $\cB([0,1])$ on $[0,1]$.

\medskip
Let $\a=(\a_d)_{d\in\cD}$ be a frequency vector satisfying $(\dagger)$. Recall that $\W_\cT(\a)$ is the set of $\w\in\W$ for which $F_{\cT,\w}(\a)$ is non-empty and that $\a_s\defeq\sum_{b\in\cB_s}\a_{(s,b)}$. We prove Theorem~\ref{Spectrum}, which asserts that $\W_\cT(\a)=\{\w\in\W:\tau_s(\w)=\a_s\,\forall s\in\cS\}$. 

\begin{proof}[Proof of Theorem~\ref{Spectrum}]
    First, suppose that $\w\in\W_\cT(\a)$. Then there is a sequence $(b_n)_{n \in \mathbb N} \in E_\omega(\alpha)$ and
    $$\tau_s(\w)=\sum_{b\in\cB_s}\tau_{(s,b)}((\w_n,b_n)_{n \in \mathbb N})=\sum_{b\in\cB_s}\a_{(s,b)}=\a_s$$
    for all $s\in\cS$, which proves the forward direction.\medskip
        
    To prove the converse statement, take $\w\in\W$ such that $\tau_s(\w)=\a_s$ for all $s\in\cS$. For each $s\in\cS$ it follows from e.g.~the proof of \cite[Proposition~4.3]{IKM23} in case $\# \mathcal B_s <+\infty$ and from e.g.~\cite[Lemma 2.1]{FLMW10} or \cite[Theorem 3.1]{Aafko} in case $\mathcal B_s = \mathbb N$ that there exists a sequence $(b_{s,n})_{n\in\N} \in \mathcal B_s^\mathbb N$ whose digit frequencies are given by the frequency vector $(\a_{(s,b)}/\a_s)_{b\in\cB_s}$:
    \begin{equation}\label{q:dsnfreq}
        \lim_{n\to+\infty}\frac{\#\{1\leq\ell\leq n\, :\, b_{s,\ell}=(s,b)\}}n=\frac{\a_{(s,b)}}{\a_s},\q\forall b\in\cB_s.
    \end{equation}
    Recall that $\tau_s(\w,n)$ denotes the number of times that the symbol $s\in\cS$ appears in the first block of $n\in\N$ symbols of $\w$. For each $n\in\N$, set $b_n\defeq b_{\w_n,\tau_{\w_n}(\w,n)}$. Then for each $(s,b)\in\cD$, it follows from the assumption $\tau_s(\w)=\a_s$ and \eqref{q:dsnfreq} that
    \begin{align*}
        \lim_{n\to+\infty}\frac{\#\{1\leq\ell\leq n\, :\, (\omega_\ell,b_\ell)=(s,b)\}}n
            &=\lim_{n\to+\infty}\frac{\#\{1\leq\ell\leq\tau_s(\w,n)\, :\, b_{s,\ell}=b\}}{\tau_s(\w,n)}\cdot\frac{\tau_s(\w,n)}n\\
            &=\frac{\a_{(s,b)}}{\a_s}\cdot \alpha_s =\a_{(s,b)},
    \end{align*}
    as desired. Therefore, $(b_n)_{n \in \mathbb N} \in E_\omega(\alpha)$ and thus $\pi_\omega((b_n)_{n \in \mathbb N}) \in F_{\mathcal T, \omega}(\alpha)$, yielding the result.
\end{proof}

For any fixed $\omega \in \Omega$, the set $[0,1]$ splits into three parts:
\begin{itemize}
\item the set $[0,1]\setminus \pi_\omega(\mathcal B_\omega^\mathbb N)$ of points that have no $(\mathcal T, \omega)$-expansion;
\item the set $\left(\bigcup_{n \in \mathbb N} \bigcup_{b_1 \cdots b_n \in \mathcal B_\omega^n} f_{\omega_1,b_1} \circ \cdots \circ f_{\omega_n,b_n}(\{0,1\}) \right) \setminus \{0,1\}$ of points that have precisely two $(\mathcal T, \omega)$-expansions;
\item the set $X_\omega = \bigcap_{n \in \mathbb N} \bigcup_{b_1 \cdots b_n \in \mathcal B_\omega^n} \langle b_1 \cdots b_n \rangle_\omega^\circ \cup (\{0,1\} \cap \pi_\omega(\mathcal B_\omega^\mathbb N))$ of points that have a unique $(\mathcal T, \omega)$-expansion.
\end{itemize}
Since $[0,1] \setminus X_\omega$ is a countable set, we have $\dim_H F_{\mathcal T, \omega}(\alpha) = \dim_H (F_{\mathcal T, \omega}(\alpha)\cap X_\omega)$. So, whenever it is more convenient we will restrict our attention to $F_{\mathcal T, \omega}(\alpha)\cap X_\omega$.

\section{Upper Bound for \texorpdfstring{$\dim_HF_{\cT,\w}(\a)$}{}}\label{UpperBoundSec}

In this section, we prove the following proposition, which is the upper bound in Theorem~\ref{dimH}.

\begin{prop}\label{p:dimHUpper}
    Let $\a=(\a_d)_{d\in\cD}$ be a frequency vector satisfying $(\dagger)$ and fix $\w\in\W_\cT(\a)$.
    Then $\dim_HF_{\cT,\w}(\a)\leq\max\{\eta_\cT,\b_\cT(\a)\}$.
\end{prop}

Throughout this section, we fix a frequency vector $\alpha = (\alpha_d)_{d \in \mathcal D}$ satisfying $(\dagger)$ and an $\w\in\W_\cT(\a)$.
We first introduce some notation and then prove four auxiliary lemmas.

\medskip
Recall that for each $m \in \mathbb N$ we have set $\cD_m = \{(s,b)\in\cD:b\leq m\}$. Additionally, for each $s\in\cS$ put $\cD_m^c\defeq\{(s,b)\in\cD:b>m\}$, $\cB_{s,m}\defeq\{b\in\cB_s:b\leq m\}$ and $\cB_{s,m}^c\defeq\{b\in\cB_s:b>m\}$. For $b_1\dotsb b_n\in\cB_\w^n$ set
$$\tau_d(\w,b_1\dotsb b_n )\defeq\#\{1\leq\ell\leq n:(\w_\ell,b_\ell)=d\}.$$
For each $n,m\in\N$ and $\varepsilon >0$, let $\cN_n=\cN_n(m,\e)$ be the set of vectors $(n_d)_{d\in\cD_m}\in\mathbb Z_{\ge 0}^{\# \cD_m}$ that satisfy the following two properties:
\begin{itemize}\parskip=2pt
\item $\left\lvert\frac{n_d}n-\a_d\right\rvert<\e$ for all $d \in \cD_m$;
\item $\sum_{b \in \mathcal B_{s,m}} n_{(s,b)} \le \tau_s(\omega,n)$ for all $s \in \mathcal S$.
\end{itemize}
The entries $n_d$ in a vector from $\mathcal N_n$ denote possible numbers of occurrences of the digits $d \in \mathcal D_m$ in a sequence $(\omega_n,b_n)_{n \in \mathbb N} \in \mathcal D^\mathbb N$, such that the frequency of the digit $d=(s,b)$ in the first $n$ terms of the sequence $(\omega_n,b_n)_{n \in \mathbb N}$ is $\varepsilon$-close to $\alpha_d$ that are compatible with $\omega$.

\begin{lem}\label{l:Nnonempty}
Let $\varepsilon>0$. Then there is an $N_{\varepsilon}\in\N$ such that $\mathcal N_n(m,\varepsilon) \neq \emptyset$ for all $n \ge N_{\varepsilon}$ and $m \in \mathbb N$.
\end{lem}

\begin{proof}
By $(\dagger)$ and Theorem~\ref{Spectrum}, there is an $N_1\in\N$ such that, for each $n \ge N_1$ and $s \in \mathcal S$, we have
\[ \tau_s(\omega,n) > n \sum_{b \in \mathcal B_s} \alpha_{(s,b)} - \frac{n\varepsilon}{2}.\]
Put $N_{\varepsilon} \defeq \max \{ N_1, \lceil \frac{2}{\varepsilon} \rceil \}$, and, for each $m \in \mathbb N$, $n \ge N_{\varepsilon}$ and $d \in \mathcal D_m$, set $n_d = \lfloor n \alpha_d - \frac{n \varepsilon}{2} \rfloor$. Then for $m \in \mathbb N$, $n \ge N_{\varepsilon}$ and $d \in \mathcal D_m$,
\[ n \alpha_d - \frac{n \varepsilon}{2}-1 < \left\lfloor n\alpha_d - \frac{n\varepsilon}{2} \right\rfloor = n_d < n \alpha_d\]
and hence
\[ \alpha_d - \varepsilon < \alpha_d - \frac{\varepsilon}{2} - \frac1n < \frac{n_d} n < \alpha_d.\]
Moreover, for any $s \in \mathcal S$, since $\# \mathcal B_{s,m} \ge 1$,
\[ \sum_{b \in \mathcal B_{s,m}} n_{(s,b)} \le \sum_{b \in \mathcal B_{s,m}} \left(n \alpha_{(s,b)} - \frac{n \varepsilon}{2} \right) \le n \sum_{b \in \mathcal B_{s,m}} \alpha_{(s,b)} - \frac{n \varepsilon}{2} < \tau_s(\omega,n).\]
Therefore, $(n_d)_{d \in \mathcal D_m} \in \mathcal N_n(m,\varepsilon)$.
\end{proof}

The next lemma concerns the sets $$H_n=H_n(\a,m,\e)\defeq \pi_\omega \left( \left\{(b_k)_{k \in \mathbb N}\in \mathcal B_\omega^\mathbb N:\left\lvert\frac{\tau_d( (\omega_k,b_k)_{k \in \mathbb N},n)}n-\a_d\right\rvert<\e\, \, \, \forall d\in\cD_m\right\} \right),$$ which will be used to cover $F_{\cT,\w}(\a)$. Note that $\mathcal N_n \neq \emptyset$ if $H_n \neq \emptyset$.

\begin{lem}\label{UpperLem1}
    Fix $t\geq0$, $\e>0$ and $m,k_0\in\N$. Then
    \begin{equation}\label{HtHn}
    \cH^t\left(\bigcap_{k\geq k_0}H_k\right)\leq\liminf_{n\to+\infty}\sum_{(n_d)\in\cN_n}\left(\prod_{d\in\cD_m}N_d^{-tn_d}\right)
        \sum_{\stackrel{b_1\cdots b_n \in \mathcal B_\omega^n}{\tau_d(\omega, b_1\cdots b_n)=n_d}} \prod_{\stackrel{1 \le \ell \le n}{b_\ell >m}} N_{\omega_\ell, b_\ell}^{-t}.
    \end{equation}
\end{lem}

\begin{proof}
    Fix $\d>0$ and let $n(\d)$ be the smallest integer such that $|\langle b_1\dotsb b_{n(\d)}\rangle_\w|^t<\d$ for all $b_1\dotsb b_{n(\d)}\in\cB_\w^{n(\d)}$.
    Then for any $n > \max\{ k_0, n(\delta)\}$, we have that
\[ \begin{split}
\cH^t_\delta\left(\bigcap_{k\geq k_0}H_k\right) =\ & \inf\left\{\sum_{U\in\cU}|U|^t\,:\cU\text{ is a }\d\text{-cover of }\bigcap_{k \ge k_0} H_k\right\}\\
%\le \ & \inf\left\{\sum_{U\in\cU}|U|^t\,:\cU\text{ is a }\d\text{-cover of } H_n\right\}\\
\le \ & \inf_{k \ge n(\delta)}\left\{\sum_{
\stackrel{b_1\cdots b_k \in \mathcal B_\omega^k}{\langle b_1\dotsb b_k\rangle_\w \cap H_n \neq \emptyset}}|\langle b_1\dotsb b_k\rangle_\w|^t\right\}
\le  \sum_{
\stackrel{b_1\cdots b_n \in \mathcal B_\omega^n}{\langle b_1\dotsb b_n\rangle_\w \cap H_n \neq \emptyset}}|\langle b_1\dotsb b_n\rangle_\w|^t\\
\le \ & \sum_{(n_d) \in \mathcal N_n} \sum_{\stackrel{b_1\cdots b_n \in \mathcal B_\omega^n}{\tau_d(\omega, b_1\cdots b_n)=n_d,d\in\cD_m}} \prod_{1 \le \ell \le n} N_{\omega_\ell, b_\ell}^{-t}\\
=\ & \sum_{(n_d) \in \mathcal N_n} \left( \prod_{d \in \mathcal D_m} N_d^{-tn_d} \right)\sum_{\stackrel{b_1\cdots b_n \in \mathcal B_\omega^n}{\tau_d(\omega, b_1\cdots b_n)=n_d, d\in\cD_m}} \prod_{\stackrel{1 \le \ell \le n}{b_\ell >m}} N_{\omega_\ell, b_\ell}^{-t}.
\end{split}\]
Since this holds for all $n > \max\{ k_0, n(\delta)\}$, we find that
\[ \cH^t_\delta\left(\bigcap_{k\geq k_0}H_k\right)  \le \liminf_{n \to +\infty} \sum_{(n_d) \in \mathcal N_n} \left( \prod_{d \in \mathcal D_m} N_d^{-tn_d} \right)\sum_{\stackrel{b_1\cdots b_n \in \mathcal B_\omega^n}{\tau_d(\omega, b_1\cdots b_n)=n_d}} \prod_{\stackrel{1 \le \ell \le n}{b_\ell >m}} N_{\omega_\ell, b_\ell}^{-t}.\]
The result follows by taking $\delta \to 0$.  
 %   and note that $n(\d)\to+\infty$ as $\d\to0$. By definition of $N_{n(\d)}$, the union over $(n_d)\in\cN_{n(\d)}$ of the collections $\cU_{n(\d)}(n_d)$ of level-$n(\d)$ FFIs forms a $\d$-cover of $H_k$ for all $k\geq k_0$. Thus, \eqref{q:lebfibcyl} yields
% \begin{align*}
 %       \cH^t\left(\bigcap_{k\geq k_0}H_k\right)
 %           &=\lim_{\d\to0}\inf\left\{\sum_{U\in\cU}|U|^t\,:\cU\text{ is a }\d\text{-cover of }H_k\text{ for all }k\geq k_0\right\}\\
  %          &\leq\lim_{\d\to0}\sum_{(n_d)\in\cN_{n(\d)}}\sum_{U\in\,\cU_{n(\d)}(n_d)}|U|^t\\
   %         &=\lim_{n\to+\infty}\sum_{(n_d)\in\cN_n}\sum_{\langle b_1\dotsb b_n\rangle_\w\in\,\cU_n(n_d)}\prod_{1\leq\ell\leq n}N_{\w_\ell,b_\ell}^{-t}\\
    %        &=\lim_{n\to+\infty}\sum_{(n_d)\in\cN_n}\sum_{\langle b_1\dotsb b_n\rangle_\w\in\,\cU_n(n_d)}\left(\prod_{\substack{1\leq\ell\leq n\\ b_\ell\leq m}}N_{\w_\ell,b_\ell}^{-t}\right)
     %           \!\left(\prod_{\substack{1\leq\ell\leq n\\ b_\ell>m}}N_{\w_\ell,b_\ell}^{-t}\right)\\
     %       &=\lim_{n\to+\infty}\sum_{(n_d)\in\cN_n}\left(\prod_{d\in\cD_m}N_d^{-tn_d}\right)
      %          \sum_{\langle b_1\dotsb b_n\rangle_\w\in\,\cU_n(n_d)}\prod_{\substack{1\leq\ell\leq n\\ b_\ell>m}}N_{\w_\ell,b_\ell}^{-t}
    %\end{align*}
    %as desired.
\end{proof}

In the next lemma, we compute the inner sum of the right-hand side of $\eqref{HtHn}$. % for each $n\in\N$ and $(n_d)\in\cN_n$. 
Given a vector $(n_d)\in\cN_n$, for each $s \in \mathcal S$, we set $n_s \defeq \sum_{b \in \mathcal B_{s,m}} n_{(s,b)}$.
%we extend it with additional elements $n_{(s,m+1)}$ for each $s \in \mathcal S$ by setting $n_{(s,m+1)}\defeq\tau_s(\w,n)-\sum_{b\in\cB_{s,m}}n_{(s,b)}$. Note that $n_{(s,m+1)}=0$ if 
%In this way, we obtain a vector $(n_d)_{d \in \mathcal D_{m+1}}$. {\color{red}What if there do not exist $(s,b)$ with $b>m$? We can let $n_{(s,m+1)}=0$ in this case, but then this does not necessarily coincide with the definition of $\mathcal D_{m+1}$.} {\color{cyan} If there do not exist $(s,b)$ with $b>m$, then $\cD_{m+1}=\cD_m$ so $n_{(s,m+1)}$ does not appear in the product over $d\in\cD_{m+1}$ below. Does this solve the issue or is there still something else?}

\begin{lem}\label{UpperLem2}
    Fix $t \ge 0$, $\varepsilon>0$ and $m\in\N$.
    Let $n$ be large enough so that $\mathcal N_n(m,\varepsilon) \neq \emptyset$, and let $(n_d)_{d \in \mathcal D_m}\in\cN_n$. Then
    \begin{multline*}\sum_{\stackrel{b_1\cdots b_n \in \mathcal B_\omega^n}{\tau_d(\omega, b_1\cdots b_n)=n_d, d \in \mathcal D_m}} \prod_{\stackrel{1 \le \ell \le n}{b_\ell >m}}N_{\omega_\ell, b_\ell}^{-t}\\
        =\left(\prod_{s\in\cS}\left(\sum_{b\in\cB_{s,m}^c}N_{s,b}^{-t}\right)^{\tau_s(\omega,n) - n_s}\right)\frac{\prod_{s\in\cS}\tau_s(\w,n)!}{\prod_{d\in\cD_m}n_d! \prod_{s \in \mathcal S}(\tau_s(\omega,n)-n_s)!}.
    \end{multline*}
\end{lem}

\begin{proof}
    Set $\k\defeq n-\sum_{d\in\cD_m}n_d$. Any string $b_1\dotsb b_n\in\cB_\w^n$ for which $\tau_d(\w,b_1\dotsb b_n)=n_d$ for all $d\in\cD_m$ has exactly $\sum_{d\in\cD_m}n_d$ terms that are at most $m$, and $\kappa$ terms that are larger than $m$. The value of the product
    \begin{equation}\label{q:Ntree}
    \prod_{\substack{1\leq\ell\leq n\\ b_\ell>m}}N_{\w_\ell,b_\ell}^{-t}
    \end{equation}
    depends only on the index and value of the $\kappa$ entries in $b_1\dotsb b_n$ that are larger than $m$.
    
    \medskip
    Fix indices $1\le v_1 < v_2 < \cdots < v_\kappa \le n$ where the digits $b>m$ can occur. This means that for each $1 \le \ell \le \kappa$, we have $\# \mathcal B_{\omega_{v_\ell}} >m$ and $n_{\omega_{v_\ell}} < \tau_{\omega_{v_\ell}} (\omega,n)$.
    In addition, fix values $c_1, \dotsc, c_\kappa$ that these digits can take; so, $c_\ell \in \mathcal B_{\omega_{v_\ell},m}^c$ for each $1 \le \ell \le \kappa$.
    Then for each $b_1 \cdots b_n \in \mathcal B_\omega^n$ that has $b_{v_\ell} = c_\ell$ for each $1 \le \ell \le \kappa$ and $\tau_d(\omega, b_1 \cdots b_n) = n_d$ for each $d \in \mathcal D_m$, we have 
    \[ \prod_{\substack{1\leq\ell\leq n\\ b_\ell>m}}N_{\w_\ell,b_\ell}^{-t} = \prod_{1 \le \ell \le \kappa} N_{\omega_{v_\ell},c_\ell}^{-t}.\]
    The number of words $b_1 \cdots b_n$ with these properties is given by
    \[ \prod_{s \in \mathcal S}\frac{n_s!}{\prod_{b \in \mathcal B_{s,m}} n_{(s,b)}!} = \frac{\prod_{s \in \mathcal S} n_s!}{\prod_{d \in \mathcal D_m} n_d!}.\]
    Therefore, if we let 
    %To prove the lemma, we therefore first group the strings $b_1\dotsb b_n$ according to the positions of the digits larger than $m$ and then further group the strings by the exact values of the digits larger than $m$.That way, the value of the product in \eqref{q:Ntree} will be equal for all strings in the same group.
    $\mathcal V_n$ be the collection of sets $\{v_1, \dotsc,v_\k\}$ of indices with $1\leq v_1<\dotsb<v_\k\leq n$, $\# \mathcal B_{\omega_{v_\ell}} >m$ and $n_{\omega_{v_\ell}} < \tau_{\omega_{v_\ell}} (\omega,n)$ for each $1 \le \ell \le \kappa$, then
    %that indicate the possible positions of terms larger than $m$ in strings $b_1\cdots b_n$ and are compatible with $\omega$. So, an index $v$ can only occur in a collection $\{v_1, \dotsc,v_\k\}$ if $\mathcal B_{\omega_v,m}^c \neq \emptyset$. Note that $\cV_n = \emptyset$ if $\k=0$. Then we have
        \begin{align}\label{CountingSum}
            \sum_{\stackrel{b_1\cdots b_n \in \mathcal B_\omega^n}{\tau_d(\omega, b_1\cdots b_n)=n_d, d \in \mathcal D_m}}\prod_{\substack{1\leq\ell\leq n\\ b_\ell>m}}N_{\w_\ell,b_\ell}^{-t}
                &= \frac{\prod_{s \in \mathcal S} n_s!}{\prod_{d \in \mathcal D_m} n_d!} \sum_{\{ v_1,\dotsc,v_\k\}\in\cV_n}\sum_{\stackrel{c_1,\dotsc,c_\k}{c_\ell\in\cB_{\w_{v_\ell},m}^c}}
                    \prod_{1 \le \ell \le \kappa} N_{\omega_{v_\ell},c_\ell}^{-t}.%\nonumber\\
                %&=\sum_{\{ v_1,\dotsc,v_\k\}\in\cV_n}\sum_{\substack{c_{v_1},\dotsc,c_{v_\k}\\ c_{v_\ell}\in\cB_{\w_{v_\ell},m}^c}}
                 %   \left(\prod_{1\le\ell\le \k}N_{\w_{v_\ell},c_{v_\ell}}^{-t}\right)\sum_{\substack{\langle b_1\dotsb b_n\rangle_\w\in\,\cU_n(n_d)\\ b_{v_\ell}=c_{v_\ell}\,\forall1 \le\ell\le \k}}1.
        \end{align}
    Next, we focus on the inner sum on the right-hand side of \eqref{CountingSum} and show that this is also independent of the choice of positions $\{v_1, \dotsc, v_\kappa\} \in \mathcal V_n$; so, fix $\{v_1,\dotsc,v_{\k}\}\in\cV_n$.
    Observe that
    \begin{align}\label{V2}
        \sum_{\substack{c_1, \dotsc, c_\kappa\\ c_\ell\in\cB_{\w_{v_\ell},m}^c}}\!\prod_{1 \le \ell \le \kappa}N_{\w_{v_\ell},b_{c_\ell}}^{-t}
            &=\sum_{c_1\in\cB_{\w_{v_1},m}^c}\!N_{\w_{v_1},c_1}^{-t}\left(\sum_{c_2\in\cB_{\w_{v_2},m}^c}\!N_{\w_{v_2},c_2}^{-t}
                \left(\dotsb\left(\sum_{c_{\k}\in\cB_{\w_{v_{\k}},m}^c}\!N_{\w_{v_{\k}},c_{\k}}^{-t}\right)\dotsb\right)\right)\nonumber\\
            &=\prod_{1 \le \ell \le \kappa}\sum_{b\in\cB_{\w_{v_\ell},m}^c}N_{\w_{v_\ell},b}^{-t},
    \end{align}
    using independence of the $\k$ sums in the second equality.
    %Note that for $n_{(s,m+1)}$ as defined before the statement of the lemma, it holds that $n_{(s,m+1)} = \# \{ 1 \le i \le \kappa \, : \, \omega_{v_i}=s\}$.
    By collecting like-terms, we get
    \begin{equation}\label{V3}
        \prod_{1 \le \ell \le \k}\sum_{b\in\cB_{\w_{v_\ell},m}^c}N_{\w_{v_\ell},b}^{-t}
            =\prod_{s\in\cS}\prod_{\substack{1\le \ell \le  \k\\ \w_{v_\ell}=s}}\sum_{b\in\cB_{s,m}^c}N_{s,b}^{-t}
            =\prod_{s\in\cS}\left(\sum_{b\in\cB_{s,m}^c}N_{s,b}^{-t}\right)^{\tau_s(\omega,n) -n_s},
    \end{equation}
    where the right-hand side is independent of $\{ v_1, \dotsc, v_\kappa\}$. Therefore, putting \eqref{CountingSum}, \eqref{V2} and \eqref{V3} together yields
    \begin{align}\label{V4}
       \sum_{\stackrel{b_1\cdots b_n \in \mathcal B_\omega^n}{\tau_d(\omega, b_1\cdots b_n)=n_d, d \in \mathcal D_m}}\prod_{\substack{1\leq\ell\leq n\\ b_\ell>m}}N_{\w_\ell,b_\ell}^{-t}
            =\left(\prod_{s\in\cS}\left(\sum_{b\in\cB_{s,m}^c}N_{s,b}^{-t}\right)^{\tau_s(\omega,n)-n_s}\right)
                \frac{\prod_{s\in\cS}n_s!}{\prod_{d\in\cD_m}n_d!}\#\cV_n.
    \end{align}
    It only remains to compute $\#\cV_n$.
    For each $s \in \mathcal S$, there are $\tau_s(\w,n)$ indices $1 \le\ell\le n$ with $\w_\ell=s$ out of which $\tau_s(\omega,n)-n_s$ have $(\w_\ell,b_\ell) = (s,b)$ for some $b >m$. There are $\binom{\tau_s(\w,n)}{n_s}$ possible arrangements of this kind and so
    \begin{align*}
        \#\cV_n=\prod_{s\in\cS}\binom{\tau_s(\w,n)}{n_s}
            &=\frac{\prod_{s\in\cS} \tau_s(\w,n)!}{\prod_{s\in\cS}n_s!\prod_{s\in\cS}(\tau_s(\omega,n) - n_s)!}.
    \end{align*}
    Substituting this into \eqref{V4} gives the result.
\end{proof}

In order to find recognisable quantities later, we shall find the exponential behaviour of some of the terms appearing in Lemmas~\ref{UpperLem1} and \ref{UpperLem2}. Set
\begin{align*}
    f & \big(t;m,\e, n,(n_d)\big)    \\
        &\defeq\frac1n\log\left[\left(\prod_{d\in\cD_m}N_d^{-tn_d}\right)\left(\prod_{s\in\cS}\left(\sum_{b\in\cB_{s,m}^c}N_{s,b}^{-t}\right)^{\tau_s(\omega,n)-n_s}\right)
            \frac{\prod_{s\in\cS}\tau_s(\w,n)!}{\prod_{d\in\cD_m}n_d! \prod_{s \in \mathcal S}(\tau_s(\omega,n)-n_s)!}\right].
\end{align*}
Then $f\big(t;m,\e,n,(n_d)\big)$ can be written as the sum of five terms.
We consider each of these separately.
% \Charlene{The notation $A_{\varepsilon,n}(n_d)$ is a bit strange, since it also depends on $m$ and $t$. Replace by $A$?}
% \medskip
First, set
\[ A\defeq t \sum_{d \in \mathcal D_m} \left( \alpha_d-\frac{n_d} n \right) \log N_d\]
so that
\[ \frac1n\log \left(\prod_{d\in\cD_m}N_d^{-tn_d}\right) = -t \sum_{d \in \mathcal D_m} \alpha_d \log N_d + A. \]
Note for each $c > 0$ and $m \in \mathbb N$ that, by taking $\varepsilon < c/(t \sum_{d \in \mathcal D_m}\log N_d)$ and $N_{\varepsilon}$ to be the value given by Lemma~\ref{l:Nnonempty}, we have for any $n \ge N_{\varepsilon}$ and $(n_d)_{d \in \mathcal D_m} \in \mathcal N_n(m,\varepsilon)$ that $|A|< c$.
% \medskip
Similarly, for the second part, set
\[ B\defeq \sum_{s \in \mathcal S} \left( \frac{\tau_s(\omega,n)-n_s} n - \sum_{b \in \mathcal B_{s,m}^c} \alpha_{(s,b)} \right) \log \sum_{b \in \mathcal B_{s,m}^c} N_{s,b}^{-t}\]
so that
\[
\frac1n \log \left(\prod_{s\in\cS}\left(\sum_{b\in\cB_{s,m}^c}N_{s,b}^{-t}\right)^{\tau_s(\omega,n)-n_s}\right)
= \sum_{s \in \mathcal S} \left( \sum_{b \in \mathcal B_{s,m}^c} \alpha_{(s,b)} \right) \log \sum_{b \in \mathcal B_{s,m}^c} N_{s,b}^{-t} + B.\]
Let $c >0$ and $m \in \mathbb N$. Take $\varepsilon < c/\sum_{s \in \mathcal S} \log \sum_{b \in \mathcal B_{s,m}^c}N_{s,b}^{-t}$. Then using $(\dagger)$ and Lemma~\ref{l:Nnonempty}, we can find an $N\in\N$ such that, for each $n \ge N$ and $(n_d)_{d \in \mathcal D_m} \in \mathcal N_n(m,\varepsilon)$,
\begin{equation}\label{q:biggerm} \left| \frac{\tau_s(\omega,n)-n_s} n - \sum_{b \in \mathcal B_{s,m}^c} \alpha_{(s,b )} \right| \le \left| \frac{\tau_s(\omega,n)} n -\sum_{b \in \mathcal B_s} \alpha_{(s,b)}\right| + \sum_{b \in \mathcal B_{s,m}} \left|  \alpha_{(s,b)} - \frac{n_{(s,b)}} n  \right| < \varepsilon
\end{equation}
and thus $|B|<c$.
% \medskip
For the last three parts, we can write
\[ \begin{split} \frac1n \log \frac{\prod_{s\in\cS}\tau_s(\w,n)!}{\prod_{d\in\cD_m}n_d! \prod_{s \in \mathcal S}(\tau_s(\omega,n)-n_s)!} =\ & \frac1n \sum_{s \in \mathcal S} \log \tau_s(\omega,n)!  - \frac1n \sum_{d \in \mathcal D_m} \log n_d!\\
& - \frac1n \sum_{s \in \mathcal S} \log (\tau_s(\omega,n)-n_s)!.
\end{split}\]
Stirling's approximation $\log x!=x\log x-x+O(\log x)$, $x\to+\infty$, and the fact that
\[ \sum_{d\in\cD_m}n_d + \sum_{s \in \mathcal S} (\tau_s(\omega,n)-n_s)=n=\sum_{s\in\cS}\tau_s(\w,n)\]
together yield
\begin{align*}
    \frac1n \log \frac{\prod_{s\in\cS}\tau_s(\w,n)!}{\prod_{d\in\cD_m}n_d! \prod_{s \in \mathcal S}(\tau_s(\omega,n)-n_s)!}
        &=\sum_{s \in \mathcal S} \frac{\tau_s(\omega,n)} n \log \tau_s(\omega,n)  - \sum_{d \in \mathcal D_m} \frac{n_d}n\log n_d\\
        &\quad\quad - \sum_{s \in \mathcal S} \frac{\tau_s(\omega,n)-n_s}n \log (\tau_s(\omega,n)-n_s) + O\left( \frac{\log n} n \right)\\
        &=\sum_{s \in \mathcal S} \frac{\tau_s(\omega,n)} n \log \frac{\tau_s(\omega,n)} n  - \sum_{d \in \mathcal D_m} \frac{n_d}n\log \frac{n_d} n\\
        &\quad\quad - \sum_{s \in \mathcal S} \frac{\tau_s(\omega,n)-n_s}n \log \frac{\tau_s(\omega,n)-n_s} n + O\left( \frac{\log n} n \right).
\end{align*}
Set
\[ C\defeq \sum_{s \in \mathcal S} \frac{\tau_s(\omega,n)} n \log \frac{\tau_s(\omega,n)} n - \sum_{s \in \mathcal S} \alpha_s \log \alpha_s,\]
and note that $|C|\to0$ as $n\to+\infty$. Set
\[ D\defeq\sum_{d \in \mathcal D_m} \alpha_d \log \alpha_d - \sum_{d \in \mathcal D_m} \frac{n_d} n\log \frac{n_d} n.\]
By uniform continuity of the map $x \mapsto x \log x$ on $[0,1]$ and by Lemma~\ref{l:Nnonempty}, for any $c>0$ and $m \in \mathbb N$, there is an $\varepsilon >0$ and an $N\in\N$ such that, for all $n \ge N$ and all $(n_d)_{d \in \mathcal D_m} \in \mathcal N_n(m,\varepsilon)$, we have $|D| < c$.
% \medskip
Finally, set
\[ E \defeq\sum_{s\in\cS}\left(\sum_{b\in\cB_{s,m}^c}\a_{(s,b)}\right)\log\left(\sum_{b\in\cB_{s,m}^c}\a_{(s,b)}\right)
                    -\sum_{s\in\mathcal S}\frac{\tau_s(\omega,n)-n_s}n\log\frac{\tau_s(\omega,n)-n_s}n. \]
Using \eqref{q:biggerm} and uniform continuity of the map $x \mapsto x \log x$ on $[0,1]$, we find again by Lemma~\ref{l:Nnonempty} that, for any $c>0$ and $m \in \mathbb N$, there is an $\varepsilon >0$ and an $N\in\N$ such that $|E| < c$ for all $n \ge N$ and all $(n_d)_{d \in \mathcal D_m} \in \mathcal N_n(m,\varepsilon)$. The above observations together lead to the following lemma.

\begin{lem}\label{UpperLem3}
    Fix $t \ge 0$.
    Then for any $m\in\N$, $\varepsilon >0$, $n \ge N_{\varepsilon}$ and $(n_d) \in \mathcal N_n (\alpha,m,\varepsilon)$,
    \begin{align*}
        f_n\big(t;m,\e,n,(n_d)\big)
            & = -t\sum_{d\in\cD_m}\a_d\log N_d+\sum_{s\in\cS}\left(\sum_{b\in\cB_{s,m}^c}\a_{(s,b)}\right)\log\sum_{b\in\cB_{s,m}^c}N_{s,b}^{-t}\\
            &\quad\quad+\sum_{s\in\cS}\a_s\log\a_s-\sum_{d\in\cD_m}\a_d\log\a_d
                -\sum_{s\in\cS}\left(\sum_{b\in\cB_{s,m}^c}\a_{(s,b)}\right)\log\left(\sum_{b\in\cB_{s,m}^c}\a_{(s,b)}\right)\\
            &\quad\quad+A+B+C+D+E+O\left(\frac{\log n}n\right).
    \end{align*}
    Moreover, for any $c>0$ and $m\in\N$ there are $\varepsilon >0$ and $\mathfrak n = \mathfrak n(m,\varepsilon)\in\N$ such that, for each $n \ge \mathfrak n$, the set $\mathcal N_n(m,\varepsilon)$ is non-empty and, for each $(n_d)_{d \in \mathcal D_m} \in \mathcal N_n(m,\varepsilon)$,
    \[ |A + B + C +  D + E|<c.\]
\end{lem}

We are now ready to prove the upper bound in Theorem~\ref{dimH}.
\begin{proof}[Proof of Proposition~\ref{p:dimHUpper}]
    Put $t_0\defeq\max\{\eta_\cT,\b_\cT(\a)\}$.
    To show that $\dim_HF_{\cT,\w}(\a)\leq t_0$, it suffices to show that $\cH^{t_0}\big(F_{\cT,\w}(\a)\big)=0$. Observe that for any fixed $\e>0$ and $m\in\N$,
    $$F_{\cT,\w}(\a)\subset\bigcup_{k_0\in\N}\bigcap_{k\geq k_0}H_k(\a,m,\e).$$
    We first show, using the lemmas above, that for any $k_0\in\N$ and appropriate $m\in\N$ and $\varepsilon >0$,
    \begin{equation}\label{Ht=0}
        \cH^{t_0}\left(\bigcap_{k\geq k_0}H_k(\a,m,\e)\right)=0.
    \end{equation}
    So, fix $k_0\in\N$. Let $\d>0$ and set $t(\d)\defeq t_0+5\d$. Then for any $m\in\N$ and $\varepsilon >0$, the Lemmas~\ref{UpperLem1}, \ref{UpperLem2} and \ref{UpperLem3} give
    \begin{equation}\label{HtTerms}
        \cH^{t(\d)}\left(\bigcap_{k\geq k_0}H_k\right)
            %&\hspace{-2cm}\leq\liminf_{n\to+\infty}\sum_{(n_d)\in\cN_n}\left(\prod_{d\in\cD_m}N_d^{-t(\d)n_d}\right)     \sum_{\langle b_1\dotsb b_n\rangle_\w\in\,\cU_n(n_d)}\prod_{\substack{1\leq\ell\leq n\\ b_\ell>m}}N_{\w_\ell,b_\ell}^{-t(\d)}\nonumber\\
           % &\hspace{-2cm}\leq\liminf_{n\to+\infty}\sum_{(n_d)\in\cN_n}\left(\prod_{d\in\cD_m}N_d^{-t(\d)n_d}\right)\!\left(\prod_{s\in\cS}\left(\sum_{b\in\cB_{s,m}^c}N_{s,b}^{-t(\d)}\right)^{\tau_s(\omega,n)-n_s}\right)\frac{\prod_{s\in\cS}\tau_s(\w,n)!}{\prod_{d\in\cD_m}n_d!\prod_{s \in \mathcal S} (\tau_s(\omega,n)-n_s)!}\nonumber\\
            \le \liminf_{n\to+\infty}\sum_{(n_d)\in\cN_n}\exp nf_n\big(t(\d);m,\e,n,(n_d)\big).
    \end{equation}
    To prove \eqref{Ht=0}, it suffices to find $\e>0$ and $m\in\N$ for which the right-hand side of \eqref{HtTerms} equals $0$.
    We start, for suitable $n$, by bounding $f_n(t(\d);m,\e,n,(n_d))$ above by a negative quantity independent of the choice of $(n_d)\in\cN_n(m,\e)$.
    By definition of $\b_\cT(\a)$, there are infinitely many $m\in\N$ such that
    $$\left\lvert\frac{\sum_{s\in\cS}\a_s\log\a_s-\sum_{d\in\cD_m}\a_d\log\a_d}{\sum_{d\in\cD_m}\a_d\log N_d}-\b_\cT(\a)\right\rvert<\d.$$
    For any such $m$,
    \begin{equation}\label{fTerm1}
        \sum_{s\in\cS}\a_s\log\a_s-\sum_{d\in\cD_m}\a_d\log\a_d\leq(\b_\cT(\a)+\d)\sum_{d\in\cD_m}\a_d\log N_d\leq(t_0+\d)\sum_{d\in\cD_m}\a_d\log N_d.
    \end{equation}
    
    % \Charlene{If $\mathcal B_{s,m}^c$ decreases, then $\sum_{b \in \mathcal B_{s,m}^c}\alpha_{(s,b)}$ approaches 0 and thus the lhs increases. (4.11) is not clear to me.}
    % \JI{$t\log t\to 0$ as $t\to0$ so each summand approaches zero} 
    
    Next, $\mathcal B_{s,m}^c \supseteq \mathcal B_{s,m+1}^c$ for each $m$ and $\bigcap_{m \in \mathbb N} \mathcal B_{s,m}^c = \emptyset$, so for each $s\in\cS$, $\sum_{b\in\cB_{s,m}^c}\a_{(s,b)}\to0$ as $m\to+\infty$.
    Since $-x\log x\to0$ as $x\to0$, and since $\mathcal D_m \subseteq \mathcal D_{m+1}$, we thus find for any large enough $m\in\N$ that
    \begin{equation}\label{fTerm2}
        -\sum_{s\in\cS}\left(\sum_{b\in\cB_{s,m}^c}\a_{(s,b)}\right)\log\left(\sum_{b\in\cB_{s,m}^c}\a_{(s,b)}\right)<\d\sum_{d\in\cD_m}\a_d\log N_d.
    \end{equation}
    In addition, since $t(\d)>\eta_\cT$, it follows that for any $s\in\cS$,
    $$\lim_{m \to +\infty} \sum_{b\in\cB_{s,m}^c}N_{s,b}^{-t(\d)}=0.$$
    Hence, for any large enough $m\in\N$ we have
    \begin{equation}\label{fTerm3}
        \sum_{s\in\cS}\left(\sum_{b\in\cB_{s,m}^c}\a_{(s,b)}\right)\log\sum_{b\in\cB_{s,m}^c}N_{s,b}^{-t(\delta)}<\d\sum_{d\in\cD_m}\a_d\log N_d.
    \end{equation}
    This means that there are infinitely many $m\in\N$, such that
    \[\begin{split}
    -t(\delta) & \sum_{d\in\cD_m}\a_d\log N_d+\sum_{s\in\cS}\left(\sum_{b\in\cB_{s,m}^c}\a_{(s,b)}\right)\log\sum_{b\in\cB_{s,m}^c}N_{s,b}^{-t(\delta)}+\sum_{s\in\cS}\a_s\log\a_s-\sum_{d\in\cD_m}\a_d\log\a_d\\
    &-\sum_{s\in\cS}\left(\sum_{b\in\cB_{s,m}^c}\a_{(s,b)}\right)\log\left(\sum_{b\in\cB_{s,m}^c}\a_{(s,b)}\right)\\
    \le\ & (-t(\delta) + t_0 + \delta + \delta + \delta)\sum_{d\in\cD_m}\a_d\log N_d = -2\delta \sum_{d\in\cD_m}\a_d\log N_d.
    \end{split}\]
    Fix such an $m$. Then by Lemma~\ref{UpperLem3} there is an $\e>0$ and an $\fn=\fn(m,\e)\in\N$ such that for all $n\geq\fn$ the set $\cN_n(m,\e)$ is non-empty and for all $(n_d)\in\cN_n(m,\e)$, 
    \begin{equation}\label{fTermABCD}
        A+B+C+D+E<\d\sum_{d\in\cD_m}\a_d\log N_d.
    \end{equation}
     %   Then for the infinitely many $m\in\N$ such that the equations~\eqref{fTerm1}, \eqref{fTerm2} and \eqref{fTerm3} hold and the $\e>0$ and $\fn\in\N$ such that \eqref{fTermABCD} holds,   we obtain
    Thus for all $n \ge \mathfrak n$,
     $$f_n\big(t(\d);m,\e,n,(n_d)\big)\leq-\d\sum_{d\in\cD_m}\a_d\log N_d+O\left(\frac{\log n}n\right),$$ which is negative for large enough $n\in\N$.
    Substituting this into \eqref{HtTerms} and noting that $$\#\cN_n(m,\e)\leq\sum_{d\in\cD_m}\bigl(n(\a_d+\e)-n(\a_d-\e)\bigr)=\#\cD_m\cdot2\e n$$ for all $n\in\N$ yields
    \begin{align}
        0\leq\cH^{t(\d)}\left(\bigcap_{k\geq k_0}H_k\right)
            &\leq\liminf_{n\to+\infty}\#\cN_n\cdot\exp\left(-\d n\sum_{d\in\cD_m}\a_d\log N_d+O(\log n)\right)\nonumber\\
            &\leq\lim_{n\to+\infty}\#\cD_m\cdot2\e n\exp\left(-\d n\sum_{d\in\cD_m}\a_d\log N_d+O(\log n)\right)=0.\label{HtBound}
    \end{align}
    Since this holds for all $k_0\in\N$ we find that
    $$\dim_H F_{\cT,\w}(\a)\le\dim_H\left(\bigcup_{k_0\in\N}\bigcap_{k\geq k_0}H_k\right)=\sup_{k_0\in\N}\dim_H\left(\bigcap_{k\geq k_0}H_k\right)\le t_0+5\d.$$
    As this holds for any $\delta>0$, we may take $\d\to0$ to obtain $\dim_HF_{\cT,\w}(\a)\le\max\{\eta_\cT,\b_\cT(\a)\}$.
\end{proof}

\section{Lower Bound: Exponent of Convergence}\label{sec:Eta}

Let $\cT$ be a collection of GLS IFSs, indexed by a finite set $\cS$, that satisfies \eqref{q:limexists}, i.e.\ such that the limit
$$\lim\limits_{n\to+\infty}\frac{\log n}{\log N_{s,n}}$$ exists for all $s\in\cS_\N$.
In this section, we will prove that $\dim_HF_{\cT,\w}(\a)\geq \eta_\cT$ for any frequency vector $\a=(\a_d)_{d\in\cD}$ satisfying $(\dagger)$ and any $\w\in\W_\cT(\a)$. If $\eta_\cT=0$, then this holds trivially; therefore, we assume hereafter that $\eta_\cT>0$, which implies that $\cS_\N\neq\emptyset$, i.e.\ that $\cT$ contains at least one system $\mathcal T_s$ with $\mathcal B_s = \mathbb N$. We prove the following proposition.

\begin{prop}\label{p:dimHLowerEta}
    Let $\a=(\a_d)_{d\in\cD}$ be a frequency vector satisfying $(\dagger)$, and fix $\w\in\W_\cT(\a)$.
    Assume that $\eta_\mathcal T >0$ and that \eqref{q:limexists} holds.
    Then $\dim_HF_{\cT,\w}(\a)\geq\eta_\cT$.
\end{prop}

Throughout this section, we fix a frequency vector $\a=(\a_d)_{d\in\cD}$ satisfying $(\dagger)$, an $\w\in\W_\cT(\a)$ and a $\vs\in\cS_\N$ such that
\begin{equation}\label{q:etaTs}
    \eta_\cT=\eta(\mathcal T_\vs)>0.
\end{equation}
To prove Proposition~\ref{p:dimHLowerEta}, we first prove several auxiliary lemmas.
\medskip

Observe that $\w\in\W_\cT(\a)$ together with the non-redundancy condition $(\dagger)$ on $\a$ implies that $\tau_\vs(\w)=\a_\vs>0$ so $\w_\ell=\vs$ for infinitely many $\ell\in\N$.
Order these indices $j_1<j_2<\dotsb<j_k<\dotsb$.
Fix $\gamma\in(1,2)$, and define a (strictly increasing) function $\th:\N\to\N$ by $\th(k)\defeq j_{\lceil k^\gamma\rceil}$.
Recall that $\tau_\vs(\w,k)$ denotes the number of occurrences of the symbol $\vs$ in the first block of length $k$ of symbols in the sequence $\w$.
Since $\tau_\vs(\w)>0$, we have $$j_k=\frac{\tau_\vs(\w,j_k)}{\tau_\vs(\w)}\big(1+o(1)\big)=\frac k{\tau_\vs(\w)}\big(1+o(1)\big),\q k\to+\infty.$$
Therefore, $\th$ has the following properties:
\begin{itemize}[leftmargin = 1.5cm]
    \item[($\Th1$)] $\w_{\th(k)}=\vs$ for all $k\in\N$;
    % \item[($\Th2$)] $\th(k)/k\to+\infty$ as $k\to+\infty$.
    \item[($\Th2$)] $\th(k)=\frac{k^\gamma}{\tau_\vs(\w)}\big(1+o(1)\big)$ as $k\to+\infty$ with $1<\gamma<2$.
\end{itemize}

To prove Proposition~\ref{p:dimHLowerEta} we will alter part of the proof of \cite[Theorem~1.2]{FLMW10} to work in our setting. Throughout this section, fix $\e,\d\in(0,1)$.
Let $\k_1(\e,\d)\in\N$ be such that
\begin{equation}\label{q:M1}
    2^k-2^{\d k}>2^{(1-\e)k},\q\forall k\geq\k_1(\e,\d).
\end{equation}
Recall that the contraction ratio $N_{s,b}^{-1}$ of the maps $f_{s,b}$ decreases when $b \in \mathcal B_s$ increases.
By \cite[p26]{PS72} and \eqref{q:limexists} we then have that
\begin{equation}\label{q:etalim}
    \eta(T_s)=\limsup_{n\to+\infty}\frac{\log n}{\log N_{s,n}}=\lim_{n\to+\infty}\frac{\log n}{\log N_{s,n}}.
\end{equation}
Thus, since $\cS_\N$ is a finite set, there is a $\k_2(\e,\d)\in\N$ such that
\begin{equation}\label{q:M2}
    N_{s,n}^{\eta(\mathcal T_s)}\le n^{1+\e},\q\forall s\in\cS_\N,\,n\geq\k_2(\e,\d).
\end{equation}
Fix $\k\geq\max\{\k_1(\e,\d),\k_2(\e,\d)\}$. In the next lemma we construct a sequence $a\in E_\w(\a)$ such that from some moment onwards all digits of $a$ correspond to contractions of GLS IFSs of $\cT$ with sufficiently large contraction ratios.  Write $\eta\defeq\min\{\eta(\mathcal T_s):s\in\cS_\N\}\in(0,1]$ and for $n \in \mathbb N$ set $\kappa(n) \defeq \min\{ k \ge \kappa \, : \, \theta(k)> n \}$.

\begin{lem}\label{l:z}
There is a sequence $a=(a_n)_{n \in \mathbb N}\in E_\omega (\alpha)$ with $\pi_\omega(a) \in X_\omega$ and for which there exists a constant $N_a\in\N$ (depending on $\e$ and $\d$) such that $N_{\w_n,a_n}\leq n^{(1+\e)/\eta}$ for all $n\geq N_a$. Furthermore, as $n\to+\infty$,
\[\sum_{1\leq\ell\leq n}\log N_{\w_\ell,a_\ell}=o(\k(n)^2).\]
\end{lem}

\begin{proof}
    \cite[Lemma 2.1]{FLMW10} yields for each $s\in\cS$ a sequence $(b_n^{(s)})_{n\in\N}\in \cB_s^\N$ such that
    \begin{equation}\label{q:z1}
        \lim_{n\to+\infty}\frac{\#\big\{1\leq\ell\leq n:b_\ell^{(s)}=b\big\}}n=\frac{\a_{(s,b)}}{\a_s},\q\forall\, b\in\cB_s,
    \end{equation}
    and $b_{\tau_s(\w,n)}^{(s)}\leq n$ for all $n\in\N$. We weave the sequences $(b_n^{(s)})_{n\in\N}$, $s\in\cS$, together according to $\w$ as we did in the proof of Theorem~\ref{Spectrum}, to obtain a sequence $\tilde a=(\tilde a_n)_{n\in\N}\in \mathcal B_\omega^\N$ by setting $\tilde a_n\defeq b_{\tau_{\w_n}(\w,n)}^{(\w_n)}$. Then $\tilde a_n \leq n$ for all $n\in\N$ and
    $$\lim_{n\to+\infty}\frac{\#\{1\leq\ell\leq n:(\w_\ell,\tilde a_\ell)=(s,b)\}}n=\frac{\a_{(s,b)}}{\a_s},\q\forall\,(s,b)\in\cD.$$
    If $\pi_\omega((\tilde a_n)_{n \in \mathbb N}) \in X_\omega$, then we take $a = \tilde a$. If this is not the case, then
\[ \pi_\omega((\tilde a_n)_{n \in \mathbb N}) \in \left(\bigcup_{n \in \mathbb N} \bigcup_{b_1 \cdots b_n \in \mathcal B_\omega^n} f_{\omega_1,b_1} \circ \cdots \circ f_{\omega_n,b_n}(\{0,1\}) \right) \setminus \{0,1\}.\]
This implies that there is an $N \in \mathbb N$ such that for each $k \ge N$ the tail $(\tilde a_{n+k})_{n \in \mathbb N}$ equals the $(\mathcal T, \sigma^k(\omega))$-expansion of 0 or 1. For any $k \in \mathbb N$ the $(\mathcal T, \sigma^k(\omega))$-expansions of 0 and 1, if they exist, either have the same tail or all digits are different. We then take $a= (a_n)_{n \in \mathbb N}$ to be the sequence given by
    \[ a_n = \begin{cases}
    \min \, (\mathcal B_{\omega_n}\setminus \{\tilde a_n\}), & \text{if } \sqrt n \in \mathbb N,\\
    \tilde a_n, & \text{otherwise.}
    \end{cases}\]
These modifications to $\tilde a$ imply that $\pi_{\sigma^k(\omega)}((a_{n+k})_{n \in \mathbb N}) \not \in \{0,1\}$ for all $k \in \mathbb N$ and thus we obtain $\pi_\omega((a_n)_{n \in \mathbb N}) \in X_\omega$. Moreover, $\tau_d((\omega_n, a_n)_{n \in \mathbb N})= \tau_d((\omega_n, \tilde a_n)_{n \in \mathbb N})$ for each $d \in \mathcal D$, so $(a_n)_{n \in \mathbb N} \in E_\omega(\alpha)$. Finally, if $a_n \neq \tilde a_n$, then $a_n \in \{1,2\}$ and thus $a_n\leq n$ for all $n \in \mathbb N$. Therefore, we can always find a sequence $a=(a_n)_{n \in \mathbb N} \in E_\omega(\alpha)$ with $\pi_\omega((a_n)_{n \in \mathbb N}) \in X_\omega$ and $ a_n\leq n$ for all $n \in \mathbb N$.
    \medskip
    
    Now, since $\#\cB_s<+\infty$ for each $s\in\cS\setminus\cS_\N$, we may define $M\defeq\max\{N_{s,b}:b\in\cB_s,s\in\cS\setminus\cS_\N\}$.
    Fix $n\geq N_a\defeq\max\{M,\k\}\in\N$.
    If $\w_n\in\cS_\N$, then $\log N_{\w_n,a_n}\leq\log N_{\w_n,n}\leq\frac{1+\e}{\eta(\mathcal T_{\w_n})}\log n$ by construction of $a$ and \eqref{q:M2}, using that $n\geq\k$. 
    If $\w_n\in\cS\setminus\cS_\N$ instead, then we have that $\log N_{\w_n,a_n}\leq\log M\leq\log n$ by definition of $M$ and $N_a$.
    The proof of the first statement then follows.\medskip

    To prove the second statement, we may apply the first statement to obtain
    \begin{equation}\label{q:z2}
    \begin{split}
        \sum_{1\leq\ell\leq n}\log N_{\w_\ell,a_\ell}
            \leq \ & \sum_{1\leq\ell\leq N_a}\log N_{\w_\ell,a_\ell}+c_a\sum_{N_a\leq\ell\leq n}\log\ell\leq C+c_a(n-N_a)\log n\\
            \le \ & (C+c_a)n\log n.
    \end{split}
    \end{equation}
    for any $n\geq N_a$, where we have set $c_a\defeq(1+\e)/\eta$ and $C\defeq\sum_{1\leq\ell\leq N_a}\log N_{\w_\ell,a_\ell}$, which are both independent of $n$. By ($\Th2$), we have (for $1<\gamma<2$) that, for any $\e_0>0$, there is a $K=K(\e_0)\in\N$ such that $|\tau_\varsigma(\omega)\theta(k)/k^\gamma-1|<\varepsilon_0$ for all $k \ge K$, which implies that
    \[ \theta(k) < \frac{k^\gamma}{\tau_\varsigma(\omega)} (1+\varepsilon_0).\]
    By the definition of $\kappa(n)$ we have $\th(k)>n$ for any $k\geq\k(n)$.
    Let $N$ be such that $\kappa(N)> K$.
    Then for all $n \ge N$ and $k \ge \kappa(n)$, $k^\gamma > \frac{\tau_\varsigma(\omega)}{1+\varepsilon_0}n$.
    Thus,
    \begin{align}\label{q:kasymp}
        \kappa(n)^2 > \left(\frac{\tau_\varsigma(\omega)}{1+\varepsilon_0}n\right)^\frac2\gamma,\q\forall n\geq N.
    \end{align}
    Since $2/\gamma>1$, we have $n\log n=o(n^{2/\gamma})$ and so it follows from \eqref{q:kasymp} that $n\log n=o(\k(n)^2)$ as $n\to+\infty$.
    Substituting this into \eqref{q:z2} completes the proof.   
\end{proof}

For any sequence $a = (a_n)_{n \in \mathbb N} \in E_\omega(\alpha)$ set
\begin{equation}\label{q:Ez}
    E_a \defeq \{ (\xi_n)_{n \in \mathbb N} \in \mathcal B_\omega^\mathbb N \, : \, \xi_{\theta(k)} \in \mathbb N \cap (2^k-2^{\delta k}, 2^k] \, \text{ for } k \ge \kappa \text{ and } \xi_\ell = a_\ell \text{ otherwise}\}.
\end{equation}
Note that ($\Th1$) implies $\pi_\omega(E_a) \neq \emptyset$ and that $\pi_\omega(E_a) \subseteq F_{\cT,\w}(\a)$ since $\th(k)/k\to+\infty$ holds by ($\Th2$). For each $a \in E_\omega(\alpha)$ we define a probability measure $\mu_a$ on $([0,1], \mathcal B([0,1]))$ by setting
\begin{equation}\label{q:muzFFI}
        \mu_a\big(\langle b_1\dotsb b_n\rangle_\w\big)= \begin{cases}
        \displaystyle \prod_{\kappa \le k < \k(n)}\big\lceil 2^{\d k}\big\rceil^{-1}, & \text{if } \langle b_1\dotsb b_n\rangle_\w \cap \pi_\omega(E_a) \neq \emptyset,\\
        0, & \text{otherwise},
        \end{cases}
    \end{equation}
for each $b_1\dotsb b_n\in\cB_\w^n$, $n\in\N$,
where $\prod_{\kappa \le k < \k(n)}\big\lceil 2^{\d k}\big\rceil^{-1}=1$ if there is no $k \ge \kappa$ with $\theta(k) \le n$. The existence of such a measure follows from the Ionescu-Tulcea Theorem (see e.g.~\cite[Theorem 14.35]{Kle20}). One easily sees that $\mu_a(\pi_\omega(E_a))=1$. To obtain the desired lower bound in the proof of Proposition~\ref{p:dimHLowerEta}, we will apply a non-autonomous analogue (Lemma~\ref{Billingsley} below) of Billingsley's Lemma, see e.g.\ \cite[Lemma~1.4.1]{BP}, to $\pi_\omega(E_a)$ and $\mu_a$.

\medskip

For $a \in E_\omega(\alpha)$ as in Lemma~\ref{l:z} we shall use \cite[p36]{Caj} (see also \cite[Satz 2]{Weg}) to compute a lower bound for the Hausdorff dimension of $\pi_\omega(E_a)$ using $\d_0$-covers of FFIs.
Dimension defined using only covers of fundamental intervals (or, more generally, cylinder sets) is known as the Billingsley dimension; see \cite{Bil65}.
In our setting, we shall define the \textit{Billingsley fibre dimension} $\dim_\w A$ of a subset $A\subset[0,1]$ to be the value $t_0\geq0$ such that $\cH_\w^t(A)=0$ for all $t>t_0$ and $\cH_\w^t(A)=+\infty$ for all $t<t_0$, where $\cH_\w^t$ is what we call the \textit{Billingsley fibre $t$-measure} and is given by $$\cH_\w^t(A)\defeq\lim_{\d_0\to0}\inf\left\{\sum_{i\in\cI}|A_i|^t:\{A_i\}_{i\in\cI}\text{ is a }\d_0\text{-cover of }A\text{ with FFIs}\right\}.$$
The limit in $\d_0$ above exists and defines a measure by e.g.\ \cite[Section 2]{Caj}.

\begin{lem}\label{Billingsley}
Let $a \in E_\w(\a)$ be as in Lemma~\ref{l:z}. If there is a $t_0\geq0$ for which
    \begin{equation}\label{BillingsleyCond1}
        \liminf_{n\to+\infty}\frac{\log\mu_a(\langle b_1\dotsb b_n\rangle_\w)}{\log|\langle b_1\dotsb b_n\rangle_\w|}\geq t_0,\quad\forall (b_n)_{n \in \mathbb N}\in E_a,
    \end{equation}
    then $\dim_H \pi_\omega(E_a)\geq t_0$.
\end{lem}

\begin{proof}
By using the FFIs $\langle b_1\dotsb b_n\rangle_\w$ in place of the usual fundamental intervals one can prove that \eqref{BillingsleyCond1} implies $\dim_\w \pi_\omega(E_a)\geq t_0$ analogously to how Billingsley's Lemma for autonomous systems is proved, (see e.g.\ \cite[Lemma~1.4.1]{BP}). Therefore, we only need to show that $\dim_H \pi_\omega (E_a)=\dim_\w \pi_\omega (E_a)$.

\medskip
Since $\pi_\omega(a) \in X_\omega$ and each $(b_n)_{n \in \mathbb N}\in E_a$ agrees with $a$ on infinitely many indices $n$, we have that $\pi_\omega: E_a \to \pi_\omega(E_a)$ is a bijection and
    \[ \pi_\omega(E_a) \subseteq X_\omega = \bigcap_{n \in \mathbb N} \bigcup_{b_1 \cdots b_n \in \mathcal B_\omega^n} \langle b_1 \cdots b_n \rangle_\omega^\circ \cap (\{0,1\} \cap \pi_\omega(\mathcal B_\omega^\mathbb N)).\]
    Since, moreover, the length of FFIs decreases to $0$ as their level increases, it suffices by \cite[p36]{Caj} to show for each $(b_n)_{n \in \mathbb N}\in E_a$ that
    \begin{equation}\label{q:Wegmann}
        \lim\limits_{n\to+\infty}\frac{\log|\langle b_1\dotsb b_{n+1}\rangle^\circ_\w|}{\log|\langle b_1\dotsb b_n\rangle_\w^\circ|}=1.
    \end{equation}
    To prove this, first recall from \eqref{q:kasymp} that there is an $N\in\N$ and a constant $c>0$ such that $\k(n)>cn^{1/\gamma}$ for all $n\geq N$, where $\gamma\in(1,2)$, so we have for some $N'\geq N$ that $n<2^{\k(n)}$ for all $n\geq N'$.
    Fix $n\geq\max\{N',N_a,\th(\k)\}$, where $N_a$ is the constant given by Lemma~\ref{l:z}. Then \eqref{q:M1}, \eqref{q:M2} and $n<2^{\k(n)}$ all hold. Also fix $(b_n)_{n \in \mathbb N}\in E_a$. By assumption, $N_{\vs,1}\leq\dotsb\leq N_{\vs,b}$ for all $b\in\N$. From $\sum_{b \in \mathcal B_\vs}N_{\vs,b}^{-1}=1$, it follows that $N_{\vs,b}\geq b$ for all $b\in\N$. Then \eqref{q:M1} implies for any $b\in(2^k-2^{\d k},2^k]$, $k\geq\k$, that
    \begin{equation}\label{q:Billingsley1}
        \log N_{\vs,b}\geq\log b>\log(2^k-2^{\d k})>(1-\e)k\log 2
    \end{equation}
    % and, in addition, (M2) yields
    % \begin{equation}\label{q:Billingsley2}
    %     \log N_{s,b}\leq\frac{1+\e}{\eta_\cT(1-\e)}\log b\leq\frac{1+\e}{\eta_\cT(1-\e)}\log m_k.
    % \end{equation}
    and so it follows from ($\Th1$) and the fact that $(b_n)_{n \in \mathbb N}\in E_a$ that
    \begin{align*}
        0\leq\frac{\log N_{\w_{n+1},b_{n+1}}}{\sum_{1\leq\ell\leq n}\log N_{\w_\ell,b_\ell}}
            &\leq\frac{\log N_{\w_{n+1},b_{n+1}}}{\sum_{\k\leq k<\k(n)}\log N_{\vs,b_{\th(k)}}}\\
            &\leq\frac1{(1-\e)\log 2}\frac{\log N_{\w_{n+1},b_{n+1}}}{\sum_{\k\leq k<\k(n)}k}\\
            &=\frac2{(1-\e)\log 2}\frac{\log N_{\w_{n+1},b_{n+1}}}{\k(n)(\k(n)-1)-\k(\k-1)}.
    \end{align*}
    We show that the quantity on the right-hand side of the above equation vanishes as $n\to+\infty$, which we do by splitting into two cases.
    First, suppose that $n+1=\th(\k(n))$.
    Then we have $b_{n+1}\in\big(2^{\k(n)}-2^{\d\k(n)},2^{\k(n)}\big]$ since $(b_n)_{n \in \mathbb N}\in E_a$ so ($\Th1)$ and \eqref{q:M2} yield
    \begin{equation}\label{q:Billingsley2}
        \log N_{\w_{n+1},b_{n+1}}
            =\log N_{\vs,b_{n+1}}\leq\frac{1+\e}{\eta_\cT}\log b_{n+1}\leq\frac{1+\e}{\eta_\cT}\k(n)\log2.
    \end{equation}
    Suppose instead that $n+1\neq\th(\k(n))$.
    Then $n+1\neq\th(k)$ for all $k\geq\k$ so it follows from the fact that $(b_n)_{n \in \mathbb N}\in E_a$ and Lemma~\ref{l:z} that
    \begin{align*}
        \log N_{\w_{n+1},b_{n+1}}=\log N_{\w_{n+1},a_{n+1}}\leq \frac{1+\e}{\eta}\log(n+1)\leq \frac{1+\e}{\eta}\k(n)\log 2,
    \end{align*}
    using that $n\geq N'$. Since $\eta_\cT\geq\eta$, we have that
    \begin{align}
        0\leq\lim_{n\to+\infty}\frac{\log N_{\w_{n+1},b_{n+1}}}{\k(n)(\k(n)-1)-\k(\k-1)}\leq\frac{1+\varepsilon}{\eta}\lim_{n\to+\infty}\frac{\k(n)\log 2}{\k(n)(\k(n)-1)-\k(\k-1)}=0,
    \end{align}
    as desired.
    Consequently,
    \begin{align*}
        \lim_{n\to+\infty}\frac{\log|\langle b_1\dotsb b_{n+1}\rangle_\w^\circ|}{\log|\langle b_1\dotsb b_n\rangle_\w^\circ|}
            &=\lim_{n\to+\infty}\frac{\sum_{1\leq\ell\leq n+1}\log N_{\w_\ell,b_\ell}}{\sum_{1\leq\ell\leq n}\log N_{\w_\ell,b_\ell}}\\
            &=1+\lim_{n\to+\infty}\frac{\log N_{\w_{n+1},b_{n+1}}}{\sum_{1\leq\ell\leq n}\log N_{\w_\ell,b_\ell}}=1
    \end{align*}
    and so \eqref{q:Wegmann} holds.
    Therefore, $\dim_H \pi_\omega(E_a)=\dim_\w \pi_\omega(E_a)\geq t_0$.
\end{proof}

We now have everything that we need to prove Proposition~\ref{p:dimHLowerEta}.

\begin{proof}[Proof of Proposition~\ref{p:dimHLowerEta}]
    Recall that $\vs\in\cS$ satisfies \eqref{q:etaTs}.
    % Let $(m_k)_{k\in\N}$ be the sequence from Lemma~\ref{l:mk}.
    Fix $\e,\d\in(0,1)$, and take $\k\in\N$ (depending on $\e,\d$) so that \eqref{q:M1} and \eqref{q:M2} hold.
    Fix $a\in E_\w(\a)$ as in Lemma~\ref{l:z}.
    Let $\mu_a$ be the measure constructed in \eqref{q:muzFFI} that is supported on $\pi_\omega(E_a)$.
    To obtain a lower bound for $\dim_HF_{\cT,\w}(\a)$, by Lemma~\ref{Billingsley}, it suffices to show that
    \begin{equation}\label{BillingsleyCond2}
        \liminf_{n\to+\infty}\frac{\log\mu_a(\langle b_1\dotsb b_n\rangle_\w)}{\log|\langle b_1\dotsb b_n\rangle_\w|}
            \geq\frac{\d}{1+\e}\eta_\cT,\quad\forall (b_n)_{n \in \mathbb N}\in E_a,
    \end{equation}
    since then $\dim_HF_{\cT,\w}(\a)\geq\dim_H \pi_\omega(E_a)\geq\frac{\d}{1+\e}\eta_\cT$ and we can take $\e\to0$ and $\d\to1$ to get the desired lower bound.
    Take $(b_n)_{n \in \mathbb N}\in E_a$ and observe that
    \begin{align*}
		-\log\left|\langle b_1\dotsb b_n\rangle_\w\right|
            &=\sum_{1\le\ell\le n}\log{N_{\w_\ell,b_\ell}}\\
		  &\leq\sum_{1\leq\ell\leq n}\log N_{\w_\ell,a_\ell}+\sum_{\k\leq k<\k(n)}\log N_{\vs,b_{\th(k)}}.
    \end{align*}
    By the same reasoning as in the derivation of \eqref{q:Billingsley2}, we have $$\log N_{\vs,b_{\th(k)}}\leq\frac{1+\e}{\eta_\cT}k\log2,\q\forall k\geq\k,$$
    so, for any $(b_n)_{n \in \mathbb N}\in E_a$,
    \begin{align*}
        \frac{\log\mu_a(\langle b_1\dotsb b_n\rangle_\w)}{\log|\langle b_1\dotsb b_n\rangle_\w|}
            &\geq\frac{\sum_{\k\leq k<\k(n)}\d k\log2}
                {\sum_{1\leq\ell\leq n}\log N_{\w_\ell,a_\ell}+\frac{1+\e}{\eta_\cT}\sum_{\k\leq k<\k(n)}k\log2}
                =\frac{c_n\d}{1+\e}\eta_\cT,
    \end{align*}
    where
    \begin{align}
        c_n\defeq\frac{\sum_{\k\leq k<\k(n)}k\log2}{\frac{\eta_\cT}{1+\e}\sum_{1\leq\ell\leq n}\log N_{\w_\ell,a_\ell}+\sum_{\k\leq k<\k(n)}k\log2}.
    \end{align}
    By Lemma~\ref{l:z}, it follows that 
\[ \lim_{n \to+\infty} \frac{\sum_{1\leq\ell\leq n}\log N_{\w_\ell,a_\ell}}{\sum_{\k\leq k<\k(n)}k\log 2}
    = \frac2{\log2}\lim_{n \to+\infty} \frac{\sum_{1\leq\ell\leq n}\log N_{\w_\ell,a_\ell}}{\k(n)(\kappa(n)-1) -\kappa(\kappa-1)} =0,\]    
and thus $c_n\to1$ as $n\to+\infty$, which gives \eqref{BillingsleyCond2}.
    Therefore, $$\dim_HF_{\cT,\w}(\a)\geq\frac{\d}{1+\e}\eta_\cT$$ by Lemma~\ref{Billingsley} and we may take $\e\to0$ and $\d\to1$ to conclude the proof.
\end{proof}

\section{Lower Bound: Fibre Dimension}\label{sec:beta1}

In this section we show that $\b_\cT(\a)$ is a lower bound for $\dim_HF_{\cT,\w}(\a)$. We will achieve this through an approximation argument on a sequence $(\cT^{(m)})_{m\in\N}$, where each $\cT^{(m)}\defeq\{\mathcal T_s^{(m)}\}_{s\in\cS}$ is a collection of finite GLS IFSs. We start by defining these systems $\mathcal T^{(m)}$, $m \in \mathbb N$.

\subsection{Approximation with finite GLS IFSs}\label{FiniteApproxSec}

For each $m\in\N$ and $s\in\cS$ there is an $R \in \mathbb N$ for which we can write
\[ [0,1]\setminus\bigcup_{b\in\cB_s,\,b\leq m}f_{s,b}([0,1]) = \bigcup_{i=1}^R (\ell_i,r_i),\]
where the intervals $(\ell_i, r_i)$, $1\le i \le R$, are disjoint and non-degenerate, whenever the set is non-empty. For $1 \le i \le R$ let $c_i \defeq \min\{ b \in \mathbb N \, : \, f_{s,b}([0,1]) \subseteq [\ell_i,r_i]\} > m$. Set $\mathcal B_s^{(m)} \defeq \{ 1, \ldots, m\} \cup \{c_i \, : \, 1 \le i \le R\}$. For $b \le m$ set $f_{s,b}^{(m)}\defeq f_{s,b}$ and for $1 \le i \le R$ set $f_{s,c_i}^{(m)}(x) \defeq (r_i-\ell_i) x + \ell_i$ and finally, put $\mathcal T_s^{(m)} \defeq \{ f_{s,b}^{(m)}\}_{b \in \mathcal B_s^{(m)}}$. See Figure~\ref{FiniteApproxFig} for examples of GLS IFSs and their first three associated approximations. For each $m\in\N$, put $\cT^{(m)}\defeq\big\{\mathcal T_s^{(m)}\big\}_{s \in \mathcal S}$, and denote by $\cD^{(m)}\defeq\big\{(s,b)\in\cD:s\in\cS,b\in\cB_s^{(m)}\big\}$ the associated digit set. Observe that $\cD_m$ is a subset of both $\cD$ and $\cD^{(m)}$; $\cD_m$ contains all $(s,b)$ for which $f_{s,b} = f_{s,b}^{(m)}$.\medskip

\begin{figure}[t]
    \centering
    \begin{subfigure}[$\mathcal T_L^{(1)}$]{
        \begin{tikzpicture}[scale=3.25]
            \draw[dashed]   (0,0.5)--(1,0.5);
            \draw[ultra thick, Bittersweet!90]  (0,0)--(1,0.5);
            \draw[ultra thick, NavyBlue!80] (0,0.5)--(1,1);
            \draw[ultra thick]  (0,0)node[below]{\small 0}--(1,0)node[below]{\small 1}--(1,1)--(0,1,0)node[left]{\small 1}--cycle;
            \draw   (0,1/2)node[left]{\scriptsize $\frac12$};
        \end{tikzpicture}}
    \end{subfigure}
    \hspace{-0.5cm}
    \begin{subfigure}[$\mathcal T_L^{(2)}$]{
        \begin{tikzpicture}[scale=3.25]
            \draw[dashed]   (0,1/3)--(1,1/3)
                            (0,1/2)--(1,1/2);
            \draw[ultra thick, Bittersweet!90] (0,0)--(1,1/3);
            \draw[ultra thick, NavyBlue!80]    (0,1/3)--(1,1/2)(0,1/2)--(1,1);
            \draw[ultra thick]  (0,0)node[below]{\small 0}--(1,0)node[below]{\small 1}--(1,1)--(0,1)node[left]{\small 1}--cycle;
            \draw   (0,1/3)node[left]{\scriptsize $\frac13$}
                    (0,1/2)node[left]{\scriptsize $\frac12$};
        \end{tikzpicture}}
    \end{subfigure}
    \hspace{-0.5cm}
    \begin{subfigure}[$\mathcal T_L^{(3)}$]{
        \begin{tikzpicture}[scale=3.25]
            \draw[dashed]   (0,1/4)--(1,1/4)
                            (0,1/3)--(1, 1/3)
                            (0,1/2)--(1, 1/2);
            \draw[ultra thick, Bittersweet!90]  (0,0)--(1,1/4);
            \draw[ultra thick, NavyBlue!80]    (0,1/4)--(1,1/3)(0,1/3)--(1,1/2)(0,1/2)--(1,1);
            \draw[ultra thick]  (0,0)node[below]{\small 0}--(1,0)node[below]{\small 1}--(1,1)--(0,1)node[left]{\small 1}--cycle;
            \draw(0,1/4)node[left]{\scriptsize $\frac14$}
                    (0,1/2)node[left]{\scriptsize $\frac12$};
        \end{tikzpicture}}
    \end{subfigure}
    \hspace{-0.5cm}
    \begin{subfigure}[$\mathcal T_L$]{
        \begin{tikzpicture}[scale=3.25]
            \draw[dashed]   (0,1/5)--(1,1/5)
                            (0,1/4)--(1,1/4)
                            (0,1/3)--(1,1/3)
                            (0,1/2)--(1,1/2);
            \draw[ultra thick, NavyBlue!80] (0,1/10)--(1,1/9)(0,1/9)--(1,1/8)(0,1/8)--(1,1/7)(0,1/7)--(1,1/6)(0,1/6)--(1,1/5)(0,1/5)--(1,1/4)(0,1/4)--(1,1/3)(0,1/3)--(1,1/2)(0,1/2)--(1,1);
            \filldraw[fill=NavyBlue!80, draw=NavyBlue!80] (0,0) rectangle (1,1/10);
            \draw[ultra thick]  (0,0)node[below]{\small 0}--(1,0)node[below]{\small 1}--(1,1)--(0,1)node[left]{\small 1}--cycle;
            \draw   
                    (0,1/3)node[left]{\scriptsize $\frac13$}
                    (0,1/2)node[left]{\scriptsize $\frac12$}
                    (0,1/7)node[left]{\scriptsize $\vdots$};
        \end{tikzpicture}}
    \end{subfigure}
    \vspace{-0.3cm}
    
    \begin{subfigure}[$\mathcal T_F^{(1)}$]{
        \begin{tikzpicture}[scale=3.25]
            \draw[dashed]   (0,1/3)--(1,1/3);
            \draw[ultra thick, Bittersweet!90]  (0,1/3)--(1,1);
            \draw[ultra thick, NavyBlue!80] (0,0)--(1,1/3);
            \draw[ultra thick]  (0,0)node[below]{\small 0}--
                                (1,0)node[below]{\small 1}--
                                (1,1)--
                                (0,1)node[left]{\small 1}--cycle;
            \draw   (0,1/3)node[left]{\scriptsize $\frac13$};
        \end{tikzpicture}}
    \end{subfigure}
    \hspace{-0.5cm}
    \begin{subfigure}[$\mathcal T_F^{(2)}$]{
        \begin{tikzpicture}[scale=3.25]
            \draw[dashed]   (0,1/3)--(1,1/3)
                            (0,1/2)--(1,1/2)
                            (0,5/6)--(1,5/6);
            \draw[ultra thick, Bittersweet!90]  (0,1/3)--(1,1/2)(0,5/6)--(1,1);
            \draw[ultra thick, NavyBlue!80] (0,0)--(1,1/3)(0,1/2)--(1,5/6);
            \draw[ultra thick]  (0,0)node[below]{\small 0}--(1,0)node[below]{\small 1}--(1,1)--(0,1)node[left]{\small 1}--cycle;
            \draw (0,1/3)node[left]{\scriptsize $\frac13$}
                    (0,1/2)node[left]{\scriptsize $\frac12$}
                    (0,5/6)node[left]{\scriptsize $\frac56$};
        \end{tikzpicture}}
    \end{subfigure}
    \hspace{-0.5cm}
    \begin{subfigure}[$\mathcal T_F^{(3)}$]{
        \begin{tikzpicture}[scale=3.25]
            \draw[dashed](0,1/3)--(1,1/3)
                (0,1/2)--(1,1/2)
                            (0,5/6)--(1,5/6);
            \draw[ultra thick, Bittersweet!90]  (0,5/6)--(1,1);
            \draw[ultra thick, NavyBlue!80] (0,0)--(1,1/3)(1,1/3)--(0,1/2)(0,1/2)--(1,5/6);
            \draw[ultra thick]  (0,0)node[below]{\small 0}--(1,0)node[below]{\small 1}--(1,1)--(0,1)node[left]{\small 1}--cycle;
            \draw(0,1/3)node[left]{\scriptsize $\frac13$}
                    (0,1/2)node[left]{\scriptsize $\frac12$}
                    (0,5/6)node[left]{\scriptsize $\frac56$};
        \end{tikzpicture}}
    \end{subfigure}
    \hspace{-0.5cm}
    \begin{subfigure}[$\mathcal T_F$]{
        \begin{tikzpicture}[scale=3.25]
            \draw[dashed]   (0,1/3)--(1,1/3)
                            (0,1/2)--(1,1/2)
                            (0,5/6)--(1,5/6);
            \draw[ultra thick, NavyBlue!80] (0,0)--(1,1/3)(1,1/3)--(0,1/2)(0,1/2)--(1,5/6)(1,5/6)--(0,1);
            \draw[ultra thick]  (0,0)node[below]{\small 0}--
                                (1,0)node[below]{\small 1}--
                                (1,1)--
                                (0,1)node[left]{\small 1}--cycle;
            \draw   (0,1/3)node[left]{\scriptsize $\frac13$}
                    (0,1/2)node[left]{\scriptsize $\frac12$}
                    (0,5/6)node[left]{\scriptsize $\frac56$};
        \end{tikzpicture}}
    \end{subfigure}
    \vspace{-0.3cm}

    \begin{subfigure}[$\mathcal T_I^{(1)}$]{
        \begin{tikzpicture}[scale=3.25]
            \draw[dashed]   (0,1/3)--(1,1/3)
                            (0,2/3)--(1,2/3);
            \draw[ultra thick, NavyBlue!80] (0,1/3)--(1,2/3);
            \draw[ultra thick, Bittersweet!90]  (0,0)--(1,1/3)(0,2/3)--(1,1);
            \draw[ultra thick]  (0,0)node[below]{\small 0}--(1,0)node[below]{\small 1}--(1,1)--
                                (0,1)node[left]{\small 1}--cycle;
            \draw   (0,1/3)node[left]{\scriptsize $\frac13$}
                    (0,2/3)node[left]{\scriptsize $\frac23$};
        \end{tikzpicture}}
    \end{subfigure}
    \hspace{-0.5cm}
    \begin{subfigure}[$\mathcal T_I^{(2)}$]{
        \begin{tikzpicture}[scale=3.25]
            \draw[dashed]   (0,1/3)--(1,1/3)
                            (0,2/3)--(1,2/3)
                            (0,1/6)--(1,1/6);
            \draw[ultra thick, NavyBlue!80] (0,1/3)--(1,2/3)(0,0)--(1,1/6);
            \draw[ultra thick, Bittersweet!90]  (0,1/6)--(1,1/3)(0,2/3)--(1,1);
            \draw[ultra thick]  (0,0)node[below]{\small 0}--
                                (1,0)node[below]{\small 1}--
                                (1,1)--
                                (0,1)node[left]{\small 1}--cycle;
            \draw   (0,1/3)node[left]{\scriptsize $\frac13$}
                    (0,2/3)node[left]{\scriptsize $\frac23$}
                    (0,1/6)node[left]{\scriptsize $\frac16$};
        \end{tikzpicture}}
    \end{subfigure}
    \hspace{-0.5cm}
    \begin{subfigure}[$\mathcal T_I^{(3)}$]{
        \begin{tikzpicture}[scale=3.25]
            \draw[dashed]   (0,1/6)--(1,1/6)
                            (0,1/3)--(1,1/3)
                            (0,2/3)--(1,2/3)
                            (0,5/6)--(1,5/6);
            \draw[ultra thick, NavyBlue!80] (0,0)--(1,1/6)(0,1/3)--(1,2/3)(1,5/6)--(0,1);
            \draw[ultra thick, Bittersweet!90]  (0,1/6)--(1,1/3)(0,2/3)--(1,5/6);
            \draw[ultra thick]  (0,0)node[below]{\small 0}--
                                (1,0)node[below]{\small 1}--
                                (1,1)--
                                (0,1)node[left]{\small 1}--cycle;
            \draw   (0,1/6)node[left]{\scriptsize $\frac16$}
                    (0,1/3)node[left]{\scriptsize $\frac13$}
                    (0,2/3)node[left]{\scriptsize $\frac23$}
                    (0,5/6)node[left]{\scriptsize $\frac56$};
        \end{tikzpicture}}
    \end{subfigure}
    \hspace{-0.5cm}
    \begin{subfigure}[$\mathcal T_I$]{
        \begin{tikzpicture}[scale=3.25]
            \draw[dashed]   (0,1/6)--(1,1/6)
                            (0,1/4)--(1,1/4)
                            (0,3/4)--(1,3/4)
                            (0,5/6)--(1,5/6);
            \draw[ultra thick, NavyBlue!80] (0,0)--(1,1/6)(0,1/6)--(1,1/4)(0,1/4)--(1,7/24)(0,7/24)--(1,15/48)
            (0,1/3)--(1,2/3)(1,5/6)--(0,1)(1,9/12)--(0,5/6)(1,17/24)--(0,9/12)(1,33/48)--(0,17/24);
            \filldraw[fill=NavyBlue!80, draw=NavyBlue!80]   (0,15/48) rectangle (1,1/3)(0,2/3) rectangle (1,33/48);
            \draw[ultra thick]  (0,0)node[below]{\small 0}--
                                (1,0)node[below]{\small 1}--
                                (1,1)--
                                (0,1)node[left]{\small 1}--cycle;
            \draw   (0,1/6)node[left]{\scriptsize $\frac16$}
                    (0,1/3)node[left]{\scriptsize $\frac13$}
                    %(0,3/12)node[left]{\tiny $\vdots$}(0,9/12)node[below]{\tiny $\vdots$}
                    (0,2/3)node[left]{\scriptsize $\frac23$}
                    (0,5/6)node[left]{\scriptsize $\frac56$};
            \draw[dashed]   (0,1/3)--(1,1/3)
                            (0,2/3)--(1,2/3);
        \end{tikzpicture}}
    \end{subfigure}
    \caption{
        Three examples of GLS IFSs $\mathcal T_s$ with their first three approximations $\mathcal T_s^{(m)}$ for $m=1,2,3$.
        In (a)-(c), we see the first three approximations of the L\"uroth system $\mathcal T_L$ shown in (d).
        In (e)-(g), we see the first three approximations of the GLS IFS $\mathcal T_F$ with finite set $\mathcal B_F$ shown in (h).
        In (i)-(k), we see the first three approximations of the GLS IFS $\mathcal T_I$ with infinite set $\mathcal B_I$ shown in (l).
        }
    \label{FiniteApproxFig}
\end{figure}
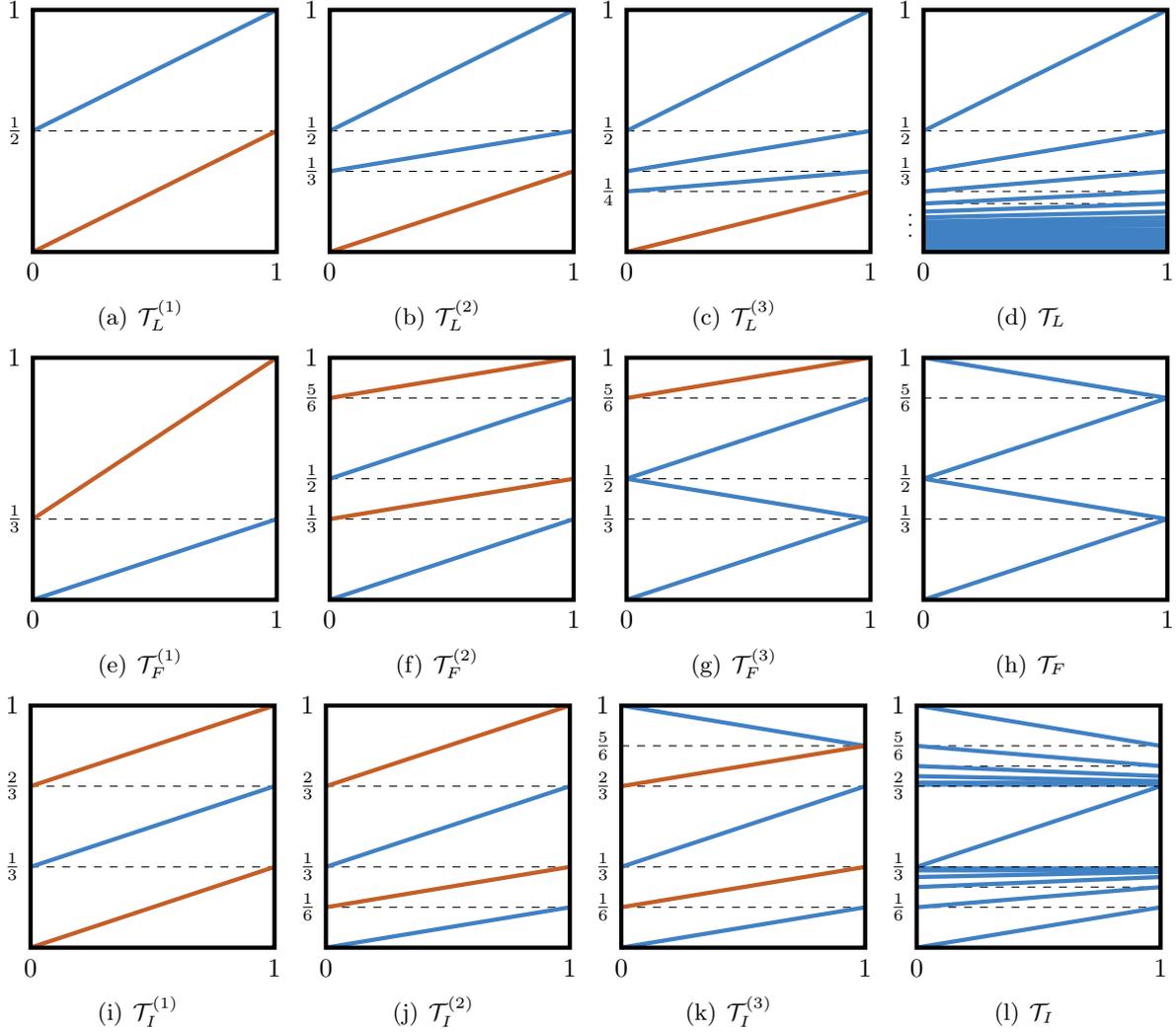

Given a frequency vector $\a=(\a_d)_{d\in\cD}$, define the frequency vector $\a^{(m)}=(\a_d^{(m)})_{d\in\cD^{(m)}}$ by
\begin{equation}\label{q:alpham}
    \a_{(s,b)}^{(m)}\defeq\left\{
    \begin{aligned}
         &\phantom{\frac12}\a_{(s,b)},  &&\text{if }b \le m,\\
         \\
         &\sum_{\substack{c\in\cB_s\\ f_{(s,c)}([0,1])\subset f_{(s,b)}^{(m)}([0,1])}}\a_{(s,c)}, &&\text{otherwise}.
    \end{aligned}\right.
\end{equation}
For $m \in \mathbb N$ and $\omega \in \Omega$ set $\mathcal B_{\omega,m}^\mathbb N = \{ (b_n)_{n \in \mathbb N} \, : \, b_n \in \mathcal B_{\omega_n}^{(m)} \, \, \forall \, n \in \mathbb N\}$. Let
\[ E_{\omega,m}(\alpha^{(m)}) = \left\{ (b_n)_{n \in \mathbb N} \in \cB_{\omega,m}^\mathbb N \, : \, \lim_{n\to+\infty}\frac{\#\{1\leq\ell\leq n:(\w_\ell,b_\ell)=d\}}n=\a_d^{(m)}\,\forall d\in\cD^{(m)} \right\}.\]
Let
\[ \pi_{\omega,m}: \cB_{\omega,m}^\mathbb N \to [0,1], \, (b_n)_{n \in \mathbb N} \mapsto  \lim_{n \to \infty} f_{\omega_1,b_1}^{(m)} \circ f_{\omega_2,b_2}^{(m)} \circ \cdots \circ f_{\omega_n,b_n}^{(m)}([0,1])\]
and write $F_{\cT^{(m)},\,\w}(\a^{(m)})\defeq \pi_\omega(E_{\omega,m}(\alpha^{(m)}))$ for the Besicovitch-Eggleston $\a^{(m)}$-level set of the system $(\cT^{(m)},\w)$. For $n \in \mathbb N$ we use the notation $\mathcal B_{\omega,m}^n = \prod_{1 \le \ell \le n} \mathcal B_{\omega_\ell}^{(m)}$ and for $b_1 \cdots b_n \in \mathcal B_{\omega,m}^n$ we denote the corresponding FFI of $(\mathcal T^{(m)}, \omega)$ by
\[ \langle b_1\dotsb b_n\rangle_{\w,m} \defeq f_{\omega_1, b_1}^{(m)} \circ \cdots \circ f_{\omega_n,b_n}^{(m)} ([0,1]).\]
Set
\[ X_\omega^{(m)} = \bigcap_{n \in \mathbb N} \bigcup_{b_1 \cdots b_n \in \mathcal B_{\omega,m}^{n}} \langle b_1\dotsb b_n\rangle_{\w,m}^\circ \cup \{0,1\}.\]

\subsection{Proving Theorem~\ref{dimH}}\label{sec:beta}

Let $\alpha = (\alpha_d)_{d \in \mathcal D}$ be a frequency vector satisfying $(\dagger)$ and fix $\omega \in \Omega_{\mathcal T}(\alpha)$.
The next step in the proof of Theorem~\ref{dimH} is to define an appropriate measure $\mu_{\w,\a}$ supported on $F_{\cT,\w}(\a)$.

\begin{lem}\label{l:muw}
    Suppose that there is an $s\in\cS$ for which $0<\a_{(s,b)}<\a_s$ for some $b\in\cB_s$. Then there is a probability measure $\mu_{\w,\a}$ on $\big([0,1],\cB([0,1])\big)$ with $\mu_{\w,\a}(F_{\cT,\w}(\a))=1$ and
    \begin{equation}\label{q:muwFFI}
        \mu_{\w,\a}\big(\langle b_1\dotsb b_n\rangle_\w\big)=\prod_{1\leq\ell\leq n}\frac{\a_{(\w_\ell,b_\ell)}}{\a_{\w_\ell}},\q\forall b_1\dotsb b_n\in\cB_\w^n,\,n\in\N.
    \end{equation}
\end{lem}

\begin{proof}
Consider the measure-theoretic ring of subsets of $\mathcal F_{\mathcal T, \omega}(\alpha)$ given by
\begin{equation}\label{q:cFw}
    \cF_{\w,\a}\defeq\big\{A\cap F_{\cT,\w}(\a):A\text{ is a countable union of FFIs of }(\cT,\w)\big\}\cup\{\emptyset\}.
\end{equation}

\medskip
For $A\cap F_{\cT,\w}(\a) \in \cF_{\w,\a}$, we define the $\cF_{\w,\a}$-maximal partition as the unique partition $\cP\subseteq \mathcal F_{\w,\a}$ of $A\cap F_{\cT,\w}(\a)$ into non-empty elements of $\cF_{\w,\a}$ such that, for each $P\in\cP$, there does not exist an FFI $B$ with $P\subsetneq B\cap F_{\cT,\w}(\a)\subsetneq A\cap F_{\cT,\w}(\a)$, and we denote it by $\mathcal P_A$.
We then set $\mu^*_{\omega, \alpha}(\emptyset)=0$ and
\begin{equation}\label{q:muw}
    \mu_{\w,\a}^*\big(A\cap F_{\cT,\w}(\a)\big)\defeq\sum_{\langle b_1\dotsb b_n\rangle_\w\cap F_{\cT,\w}(\a)\in\cP_A}\prod_{1\leq\ell\leq n}\frac{\a_{(\w_\ell,b_\ell)}}{\a_{\w_\ell}}
\end{equation}
for any countable union of FFIs $A$.
Note that if $A = \langle b_1 \cdots b_n \rangle_\omega$ is an FFI, then we have $\mathcal P_A = \{ A \cap F_{\cT,\w}(\a)\}$ and
$\mu_{\w,\a}^*(A\cap F_{\cT,\w}(\a)) = \prod_{1\leq\ell\leq n}\a_{(\w_\ell,b_\ell)}/\a_{\w_\ell} <+\infty$ as in \eqref{q:muwFFI}.

\medskip
If there is an $s\in\cS$ for which $0<\a_{(s,b)}<\a_s$ for some $b\in\cB_s$, then for any $b_1 \cdots b_n \in \mathcal B_\omega^n$ we have
\[ \mu^*_{\omega,\alpha} \big( \langle b_1 \cdots b_n\rangle_\omega \cap F_{\cT,\w}(\a)\big) = \prod_{1 \le \ell \le n} 
\frac{\alpha_{(\omega_\ell, b_\ell)}}{\alpha_{\omega_\ell}} \le \left( \frac{\max\{ \alpha_{(s,b)}:b \in \mathcal B_s \}}{\alpha_s} \right)^{\tau_s(\omega,n)}.\]
Since $\omega \in \Omega_\mathcal T(\alpha)$, we have by Theorem~\ref{Spectrum} that $\lim_{n \to+\infty} \tau_s(\omega,n)=+\infty$. Hence, we find that for each $\varepsilon >0$ there is a $K \in \mathbb N$ such that for all $n \ge K$ and for all $b_1 \cdots b_n \in \mathcal B_\omega^n$,
\begin{equation}\label{q:measureofpoints} \mu^*_{\omega,\alpha} \big( \langle b_1 \cdots b_n\rangle_\omega \cap F_{\cT,\w}(\a)\big) < \varepsilon.
\end{equation}
The proof that $\mu_{\omega, \alpha}^*$ is a pre-measure that can be extended to a probability measure $\mu_{\omega, \alpha}$ now follows from the Carath\'eodory extension theorem using standard arguments. By setting $$\mu_{\w,\a}(B)=\mu_{\w,\a}\big(B\cap F_{\cT,\w}(\a)\big),\q\forall B\in\cB([0,1]),$$
we can then further extend $\mu_{\w,\a}$ to the measurable space $\big([0,1],\cB([0,1])\big)$.
\end{proof}

Lemma~\ref{l:muw} holds for any finite family of GLS IFSs. We can therefore apply it to obtain a measure $\mu_{\omega,\alpha}^{(m)}$ satisfying 
\begin{equation}\label{q:muwmFFI}
        \mu_{\w,\a}^{(m)}\big(\langle b_1\dotsb b_n\rangle_{\w,m}\big)
            =\prod_{1\leq\ell\leq n}\frac{\a_{(\w_\ell,b_\ell)}^{(m)}}{\a_{\w_\ell}},\hspace{1cm}\forall b_1\dotsb b_n\in\cB_{\w,m}^n
    \end{equation}
for any $\mathcal T^{(m)}$, $m \in \mathbb N$. The next lemma relates the measures $\mu_{\w,\a}$ and $\mu_{\w,\a}^{(m)}$, $m\in\N$.

\begin{lem}\label{LimMeasure}
    For every $B\in\cB([0,1])$, $$\mu_{\w,\a}(B)=\lim\limits_{m\to+\infty}\mu_{\w,\a}^{(m)}(B).$$
\end{lem}

\begin{proof}
    As $\cF\cup\{\emptyset\}$ generates $\cB([0,1])$ and $\mu_{\w,\a}(\emptyset)=\mu_{\w,\a}^{(m)}(\emptyset)=0$ for all $m\in\N$, it suffices to show for each $b_1\dotsb b_n\in\cB_\w^n$ that there is an $m_0\in\N$ such that
    $$\mu_{\w,\a}^{(m)}\big(\langle b_1\dotsb b_n\rangle_\w\big)=\mu_{\w,\a}\big(\langle b_1\dotsb b_n\rangle_\w\big)$$
    for all $m\geq m_0$.
    So, fix $n\in\N$ and $b_1\dotsb b_n\in\cB_\w^n$.
    Then $m_0\defeq\max\{b_\ell:1\leq\ell\leq n\}+1$ is such that $(\w_\ell,b_\ell)\in\cD_m\subset\cD^{(m)}$ for all $1\leq\ell\leq n$ and all $m\geq m_0$ so, by \eqref{q:muwFFI}, \eqref{q:muwmFFI} and the definition of $\a^{(m)}$, we have
    \begin{align*}
        \mu_{\w,\a}^{(m)}\big(\langle b_1\dotsb b_n\rangle_\w\big)
            &=\mu_{\w,\a}^{(m)}\big(\langle b_1\dotsb b_n\rangle_{\w,m}\big)=\prod_{1\leq\ell\leq n}\frac{\a_{(\w_\ell,b_\ell)}^{(m)}}{\a_{\w_\ell}}=\prod_{1\leq\ell\leq n}\frac{\a_{(\w_\ell,b_\ell)}}{\a_{\w_\ell}}
                =\mu_{\w,\a}\big(\langle b_1\dotsb b_n\rangle_\w\big)
    \end{align*}
    for all $m\geq m_0$.
\end{proof}

To prove that $\dim_HF_{\cT,\w}(\a)\geq\b_\cT(\a)$ for all $\w\in\W_\cT(\a)$, we use the systems  $(\cT^{(m)},\w)$, $m \in \mathbb N$, and we begin with the case when $\sum_{d\in\cD}\a_d\log N_d<+\infty$.

\begin{prop}\label{p:dimHLowerBeta}
	If $\sum_{d\in\cD}\a_d\log N_d<+\infty$, then $\dim_HF_{\cT,\w}(\a)\geq\b_\cT(\a)$ for all $\w\in\W_\cT(\a)$.
\end{prop}

\begin{proof}
    Fix $\w\in\W_\cT(\a)$.
    We first prove a trivial case.
    If there is no $s\in\cS$ for which $0<\a_{(s,b)}<\a_s$ for some $b\in\cB_s$, then it follows for each $s\in\cS$ that there is a $c_s\in\cB_s$ such that $\a_{(s,c_s)}=\a_s$ and $\a_{(s,b)}=0$ for all $b\in\cB_s\setminus\{c_s\}$.
    Thus, we have $$\sum_{s\in\cS}\a_s\log\a_s=\sum_{s\in\cS}\a_{(s,c_s)}\log\a_{(s,c_s)}=\sum_{d\in\cD}\a_d\log\a_d,$$ using the convention that $0\log0\defeq0$, which implies that
    $$\b_\cT(\a)=\liminf_{m\to+\infty}\frac{\sum_{s\in\cS}\a_s\log\a_s-\sum_{d\in\cD}\a_d\log\a_d}{\sum_{d\in\cD_m}\a_d\log N_d}=0\leq\dim_HF_{\cT,\w}(\a).$$
    Therefore, suppose that there is an $s\in\cS$ such that $0<\a_{(s,b)}<\a_s$ for some $\tilde b\in\cB_s$. %This implies that there is an $M\in\N$ such that, for each $m\geq M$, there is a $b\in\cB_s^{(m)}$ with $0<\a_{(s,b)}^{(m)}<\a_s$.
    This implies $0<\a_{(s,\tilde b)}^{(m)} = \a_{(s,\tilde b)}<\a_s$ for all $m > \tilde b$. Let $\mu_{\w,\a}$ be the measure defined in Lemma~\ref{l:muw}, and, for each $m > \tilde b$, let $\mu_{\w,\a}^{(m)}$ be the measure defined in \eqref{q:muwmFFI}. Fix $\d>0$.
    Since $\mu_{\w,\a}\big(F_{\cT,\w}(\a)\big)=1$, Lemma~\ref{LimMeasure} implies that there is an $M(\d) > \tilde b$ such that $\mu_{\w,\a}^{(m)}\big(F_{\cT,\w}(\a)\big)>1-\d$ for all $m\geq M(\d)$.
    Fix $m\geq M(\d)$.
    Then
    \begin{align}\label{q:LowerBoundBeta1}
        \dim_HF_{\cT,\w}(\a)
            %&\geq\inf\{\dim_HA:A\subset[0,1],\,\mu_{\w,\a}(A)=1\}\nonumber\\
            &\geq\inf\left\{\dim_HA:A\in \mathcal B([0,1]),\,\mu_{\w,\a}^{(m)}(A)>1-\d\right\}\geq\dim_H\mu_{\w,\a}^{(m)}-E(\d),
    \end{align}
    where
    \[ E(\d)\defeq\sup_{m'\geq M(\d)}\,\left|\dim_H\mu_{\w,\a}^{(m')}-\inf\bigl\{\dim_HA:A\in \mathcal B([0,1]),\,\mu_{\w,\a}^{(m')}(A)>1-\d\bigr\}\right|.\]
From \eqref{q:dimHmu} we see for any $m' \ge M(\delta)$ that
\[ \dim_H\mu_{\omega, \alpha}^{(m')}=\lim_{\e\to0}\inf\{\dim_HA:A\in \mathcal B([0,1]),\,\mu_{\omega,\alpha}^{(m')}(A)>1-\e\}.\]
Then
\[\lim_{\e\to0}\,\left|\dim_H\mu_{\w,\a}^{(m')}-\inf\bigl\{\dim_HA:A\in \mathcal B([0,1]),\,\mu_{\w,\a}^{(m')}(A)>1-\e\bigr\}\right|=0\]
so, since $M(\d)$ increases to $+\infty$ as $\d$ decreases to $0$, we see that $E(\d)$ is monotone decreasing to $0$ as $\d\to0$.

\medskip
From \eqref{q:measureofpoints} and \eqref{q:muwmFFI}, together with the fact that $[0,1]\setminus X_\omega^{(m)}$ is a countable set, it follows that $\mu_{\w,\a}^{(m)} \big(F_{\cT^{(m)},\, \w}(\a^{(m)}) \cap X_\omega^{(m)}\big)=1$. Recall the definition of $f_{s,b}^{(m)}$ and set $N_{s,b}^{(m)} = N_{s,b}$ for $b \le m$ and $N_{s,c_i}^{(m)}= (r_i-\ell_i)^{-1}$ for $1 \le i \le R$. By a standard argument in dimension theory (see e.g.\ \cite[Theorem~15.3]{Pes97}), replacing the balls in the definition of the (lower) pointwise dimension of $\mu_{\w,\a}^{(m)}$ at $x$ with FFIs of $(\cT^{(m)},\w)$ yields a lower bound for $x \in F_{\cT^{(m)},\, \w}(\a^{(m)}) \cap X_\omega^{(m)}$:
    \begin{align*}
        \underline{d}_{\mu_{\w,\a}^{(m)}}(x)
            &\defeq\liminf_{r\to0}\frac{\log\mu_{\w,\a}^{(m)}\big(B(x,r)\big)}{\log r}
                \geq\liminf_{n\to+\infty}\frac{\log\mu_{\w,\a}^{(m)}\big(\big\langle b_1^{(m)}\dotsb b_n^{(m)}\big\rangle_{\w,m}\big)}
                    {\log\bigl|\big\langle b_1^{(m)}\dotsb b_n^{(m)}\big\rangle_{\w,m}\bigr|}\\
            &\,=\liminf_{n\to+\infty}\frac{\frac1n\sum_{1\leq\ell\leq n}\log(\a_{(\w_\ell,b_\ell^{(m)})}^{(m)}/\a_{\w_\ell})}
                {-\frac1n\sum_{1\leq\ell\leq n}\log N_{\w_\ell,b_\ell^{(m)}}^{(m)}},
    \end{align*}
    where $(b_n^{(m)})_{n \in \mathbb N} = \pi_{\omega,m}^{-1}(x)$. Writing $\tau_d(\w,b_1^{(m)}\dotsb b_n^{(m)} )\defeq\#\{1\leq\ell\leq n:(\w_\ell,b_\ell^{(m)})=d\}$ as before, we may collect like terms to see that for such $x$
    \begin{align*}
        \lim_{n\to+\infty}\frac1n\sum_{1\leq\ell\leq n}\log\frac{\a_{(\w_\ell,b_\ell^{(m)})}^{(m)}}{\a_{\w_\ell}}  &
        =\lim_{n\to+\infty}\left(\sum_{d\in\cD^{(m)}}\frac{\tau_d\big(\w,b_1^{(m)}\dotsb b_n^{(m)}\big)}n\log\a_d^{(m)}
                -\sum_{s\in\cS}\frac{\tau_s(\w,n)}n\log\a_s\right)\\
            &=\sum_{d\in\cD^{(m)}}\a_d^{(m)}\log\a_d^{(m)}-\sum_{s\in\cS}\a_s\log\a_s
    \end{align*}
    and
    \begin{align*}
        \lim_{n\to+\infty}\frac1n\sum_{1\leq\ell\leq n}\log N_{\w_\ell,b_\ell}^{(m)}
            &=\lim_{n\to+\infty}\sum_{d\in\cD^{(m)}}\frac{\tau_d\big(\w,b_1^{(m)}\dotsb b_n^{(m)}\big)}n\log N_d^{(m)}\\
        &=\sum_{d\in\cD^{(m)}}\a_d^{(m)}\log N_d^{(m)} >0.
    \end{align*}
Since the sets $\cD^{(m)}$ and $\cS$ are finite, we therefore have for $\mu_{\w,\a}^{(m)}$-a.e.\ $x\in[0,1]$ that
    \begin{align*}
        \underline{d}_{\mu_{\w,\a}^{(m)}}(x)
            &\geq\frac{\sum_{s\in\cS}\a_s\log\a_s-\sum_{d\in\cD^{(m)}}\a_d^{(m)}\log\a_d^{(m)}}{\sum_{d\in\cD^{(m)}}\a_d^{(m)}\log N_d^{(m)}}\\
            &\geq\frac{\sum_{s\in\cS}\a_s\log\a_s-\sum_{d\in\cD_m}\a_d\log\a_d}{\sum_{d\in\cD_m}\a_d\log N_d+\sum_{d\in\cD^{(m)}\setminus\cD_m}\a_d^{(m)}\log N_d^{(m)}}\\
            &=e_m\frac{\sum_{s\in\cS}\a_s\log\a_s-\sum_{d\in\cD_m}\a_d\log\a_d}{\sum_{d\in\cD_m}\a_d\log N_d},
            %&=\frac{\sum_{s\in\cS}\a_s\log\a_s-\sum_{d\in\cD}\a_d\log\a_d}{\sum_{d\in\cD}\a_d\log N_d},\\
    \end{align*}
    where
    \begin{align*}
        e_m\defeq\frac{\sum_{d\in\cD_m}\a_d\log N_d}{\sum_{d\in\cD_m}\a_d\log N_d+\sum_{d\in\cD^{(m)}\setminus\cD_m}\a_d^{(m)}\log N_d^{(m)}}.
    \end{align*}
Putting this together with \eqref{q:LowerBoundBeta1} and Lemma~\ref{l:Pes} yields
    \begin{align*}
        \dim_HF_{\cT,\w}(\a)\geq e_m\frac{\sum_{s\in\cS}\a_s\log\a_s-\sum_{d\in\cD_m}\a_d\log\a_d}{\sum_{d\in\cD_m}\a_d\log N_d}-E(\d)
    \end{align*}
    for all $m\geq M(\d)$.
    Now, for any $m \ge M(\delta)$, observe that $N_{s,c}\geq N_{s,b}^{(m)}$ for any $(s,b)\in\cD^{(m)}$ and $(s,c)\in\cD$ such that $f_{s,c}([0,1])\subset f_{s,b}^{(m)}([0,1])$ so the definition of $\a^{(m)}$ yields
    \begin{align*}
        0\leq\sum_{d\in\cD^{(m)}\setminus\cD_m}\a_d^{(m)}\log N_d^{(m)}
            &=\sum_{s\in\cS}\sum_{\substack{b\in\cB_s^{(m)}\\ b>m}}\sum_{\substack{c\in\cB_s\\ f_{s,c}([0,1])\subset f_{s,b}^{(m)}([0,1])}}\a_{(s,c)}\log N_{(s,b)}^{(m)}\\
            &\leq\sum_{s\in\cS}\sum_{\substack{b\in\cB_s^{(m)}\\ b>m}}\sum_{\substack{c\in\cB_s\\ f_{s,c}([0,1])\subset f_{s,b}^{(m)}([0,1])}}\a_{(s,c)}\log N_{(s,c)}\\
            &=\sum_{d\in\cD\setminus\cD_m}\a_d\log N_d<+\infty
    \end{align*}
    by the assumption.
    Thus, we may take $m\to+\infty$ to see that $\sum_{d\in\cD^{(m)}\setminus\cD_m}\a_d^{(m)}\log N_d^{(m)}\to0$ and so $e_m\to1$ as $m\to+\infty$. Hence, taking $m\to+\infty$ gives
    $$\dim_HF_{\cT,\w}(\a)\geq\liminf_{m\to+\infty}\left(\frac{\sum_{s\in\cS}\a_s\log\a_s-\sum_{d\in\cD_m}\a_d\log\a_d}{\sum_{d\in\cD_m}\a_d\log N_d}-E(\d)\right)
        =\b_\cT(\a)-E(\d).$$
    Since $E(\d)\to0$ as $\d\to0$, we may take $\d\to0$ to obtain the desired lower bound.
\end{proof}

\begin{proof}[Proof of Theorem~\ref{dimH}]
    The upper bound is precisely Proposition~\ref{p:dimHUpper}.\medskip
    
    In the case that $\sum_{d\in\cD}\a_d\log N_d<+\infty$, then the lower bound follows from Proposition~\ref{p:dimHLowerEta} and Proposition~\ref{p:dimHLowerBeta}.
    If $\sum_{d\in\cD}\a_d\log N_d=+\infty$, then we have for all $\w\in\W_\cT(\a)$ that
    \begin{equation}\label{LowerbetaEq}
        \dim_HF_{\cT,\w}(\a)\geq\eta_\cT\geq\b_\cT(\a).
    \end{equation}
    The first inequality follows from Proposition~\ref{p:dimHLowerEta} and the second inequality follows analogously to \cite[Lemma 2.4]{FLMW10}.
    In other words, we have $\max\{\eta_\cT,\b_\cT(\a)\}=\eta_\cT$ so Proposition~\ref{p:dimHLowerEta} gives the desired lower bound in this case.\medskip
    
    The proof of the final statement of Theorem~\ref{dimH} follows from Proposition~\ref{p:dimHUpper} and \eqref{LowerbetaEq}.
\end{proof}

\newcommand{\etalchar}[1]{$^{#1}$}
\def\cprime{$'$}

\end{document}